\documentclass[dvipdfmx]{article}
\usepackage{amsthm, amscd, amsmath, amsfonts, mathrsfs}
\newtheorem{theo}{Theorem}[section]
\newtheorem{proposition}[theo]{Proposition}

\newtheorem{rem}[theo]{Remark}
\newtheorem{definition}[theo]{Definition}

\input xy
\xyoption{all}

\begin{document}

\title{On a B-field Transform of Generalized Complex Structures over Complex Tori}

\author{Kazushi Kobayashi\footnote{Department of Mathematics, Faculty of Education, Mie University, 1577 Kurimamachiya-cho, Tsu city, Mie, 514-8507, Japan. E-mail: kobayashi@edu.mie-u.ac.jp. 2020 Mathematics Subject Classification: 14J33, 53D37, 53D18 (primary), 14F08, 53C08 (secondary). Keywords: torus, generalized complex geometry, gerbe, homological mirror symmetry.}}

\date{}

\maketitle

\maketitle

\begin{abstract}
Let $(X^n,\check{X}^n)$ be a mirror pair of an $n$-dimensional complex torus $X^n$ and its mirror partner $\check{X}^n$. Then, by SYZ transform, we can construct a holomorphic line bundle with an integrable connection from each pair of a Lagrangian section of $\check{X}^n\to \mathbb{R}^n/\mathbb{Z}^n$ and a unitary local system along it, and those holomorphic line bundles with integrable connections forms a dg-category $DG_{X^n}$. In this paper, we focus on a certain B-field transform of the generalized complex structure induced from the complex structure on $X^n$, and interpret it as the deformation $X_{\mathcal{G}}^n$ of $X^n$ by a flat gerbe $\mathcal{G}$. Moreover, we construct the deformation of $DG_{X^n}$ associated to the deformation from $X^n$ to $X_{\mathcal{G}}^n$, and also discuss the homological mirror symmetry between $X_{\mathcal{G}}^n$ and its mirror partner on the object level.
\end{abstract}

\tableofcontents

\section{Introduction}
\subsection{Background}
Let $X^n$ be an $n$-dimensional complex torus and $\check{X}^n$ be its mirror partner. The aim of this paper is two-fold: Firstly, to understand the deformation $X_{\mathcal{G}}^n$ of $X^n$ by a certain flat gerbe $\mathcal{G}$ and the corresponding deformation $\check{X}_{\mathcal{G}}^n$ of $\check{X}^n$ via generalized complex geometry. Secondary, to discuss the homological mirror symmetry for the deformed mirror pair $(X_{\mathcal{G}}^n, \check{X}_{\mathcal{G}}^n)$ on the object level. In this subsection, in order to explain the details of these aims, we review the homological mirror symmetry briefly. 

The homological mirror symmetry conjecture \cite{Kon} was originally proposed for Calabi-Yau manifolds by Kontsevich in 1994, and it states the following: for each Calabi-Yau manifold $M$, there exists a Calabi-Yau manifold $\check{M}$ such that there exists an equivalence
\begin{equation*}
D^b(Coh(M))\cong Tr(Fuk(\check{M}))
\end{equation*}
as triangulated categories. Here, $D^b(Coh(M))$ is the bounded derived category of coherent sheaves over $M$ and $Tr(Fuk(\check{M}))$ is the enhanced triangulated category of the Fukaya category $Fuk(\check{M})$ over $\check{M}$ \cite{Fukaya category} in the sense of Bondal-Kapranov-Kontsevich construction \cite{bondal, Kon}. In general, when we regard the mirror pair $(X^n, \check{X}^n)$ as the trivial special Lagrangian torus fibrations $X^n\to \mathbb{R}^n/\mathbb{Z}^n$ and $\check{X}^n\to \mathbb{R}^n/\mathbb{Z}^n$ on the same base space $\mathbb{R}^n/\mathbb{Z}^n$ according to SYZ construction \cite{SYZ}, we can associate a holomorphic line bundle with an integrable connection $E(s,\mathcal{L}^{\nabla})\to X^n$ to each pair $(s,\mathcal{L}^{\nabla})$ of a Lagrangian section $s$ of $\check{X}^n\to \mathbb{R}^n/\mathbb{Z}^n$ and a unitary local system $\mathcal{L}^{\nabla}$ along it. Moreover, those holomorphic line bundles with integrable connections $E(s,\mathcal{L}^{\nabla})$ naturally forms a dg-category $DG_{X^n}$ (see \cite{kajiura, exact} for instance), and we expect that this dg-category $DG_{X^n}$ generates the bounded derived category of coherent sheaves $D^b(Coh(X^n))$ over $X^n$ in the sense of Bondal-Kapranov-Kontsevich construction (at least, it is known that it split generates $D^b(Coh(X^n))$ when $X^n$ is an abelian variety \cite{orlov, abouzaid}). Hence, to consider the dg-category $DG_{X^n}$ is valid in the study of the homological mirror symmetry for $(X^n, \check{X}^n)$.

So far, various studies of the homological mirror symmetry for tori have been done since one of the most fundamental examples of mirror pairs is a pair of tori (\cite{elliptic, dg, Fuk, A-inf, abouzaid} etc.), and the approach via generalized complex geometry is one of them. Generalized complex geometry was introduced by Hitchin and Gualtieri in \cite{hitchin, Gual}, and this provides a way of treating complex structures and symplectic structures uniformly. In particular, Oren Ben-Bassat intensively studies the mirror symmetry from the viewpoint of generalized complex geometry in \cite{part1, part2}, and the framework of generalized complex geometry is also used in our previous studies of the homological mirror symmetry for tori \cite{kazushi, gerby}.

\subsection{Motivation}
First, let us recall the generalized complex structure induced from the complex structure on $X^n$ and B-field transforms of it. We assume that the period matrix of $X^n$ is given by a regular\footnote{This regularity assumption is not necessary when we define an $n$-dimensional complex torus $X^n=\mathbb{C}^n/\mathbb{Z}^n\oplus T\mathbb{Z}^n$. This is necessary in order to define a mirror partner of $X^n$ (see \cite{kazushi}).} matrix $T\in M(n;\mathbb{C})$ such that the imaginary part of $T$ is positive definite, i.e., $X^n=\mathbb{C}^n/\mathbb{Z}^n\oplus T\mathbb{Z}^n$, where $M(n;\mathbb{C})$ denotes the set of complex matrices of order $n$ (we use the notations $M(n;\mathbb{R})$ and $M(n;\mathbb{Z})$ in this sense). We sometimes decompose the matrix $T$ into its real part $X\in M(n;\mathbb{R})$ and its imaginary part $Y\in M(n;\mathbb{R})$: $T=X+\mathbf{i}Y$, where $\mathbf{i}:=\sqrt{-1}$. For this $X^n$, we can take an open covering $\{ O_i \}_{i\in I}$ such that the open sets $O_i$ and their intersections are all contractible, and we locally express the complex coordinates of $O_i$ ($i\in I$) as $x+Ty$, where
\begin{equation*}
x:=\left( \begin{array}{ccc} x_1 \\ \vdots \\ x_n \end{array} \right), \ y:=\left( \begin{array}{ccc} y_1 \\ \vdots \\ y_n \end{array} \right)\in \mathbb{R}^n.
\end{equation*}
Let us consider the generalized complex structure $\mathcal{I}_T : \Gamma(TX^n\oplus T^*X^n)\to \Gamma(TX^n\oplus T^*X^n)$ induced from the complex structure on $X^n$, where $TX^n$, $T^*X^n$ and $\Gamma(TX^n\oplus T^*X^n)$ denote the tangent bundle on $X^n$, the cotangent bundle on $X^n$ and the space of smooth sections of $TX^n\oplus T^*X^n$, respectively. The representation matrix of $\mathcal{I}_T$ with respect to the basis $\mathcal{B}:=\{ \frac{\partial}{\partial x_1},\cdots, \frac{\partial}{\partial x_n}, \frac{\partial}{\partial y_1},\cdots, \frac{\partial}{\partial y_n}, dx_1,\cdots, dx_n, dy_1,\cdots, dy_n \}$ of $\Gamma(TX^n\oplus T^*X^n)$ is described by
\begin{equation*}
\left( \begin{array}{cc} I_T & O \\ O & -I_T^t \end{array} \right), \ \ \ I_T:=\left( \begin{array}{cc} -XY^{-1} & -Y-XY^{-1}X \\ Y^{-1} & Y^{-1}X \end{array} \right),
\end{equation*}
where $A^t$ denotes the transpose of a given matrix $A$. In this setting, by using two arbitrary alternating matrices $B_1$, $B_3\in M(n;\mathbb{R})$ and an arbitrary matrix $B_2\in M(n;\mathbb{R})$, we consider the closed 2-form 
\begin{equation}
\pi \sum_{i,j=1}^n \Bigl( (B_1)_{ij}dx_i\wedge dx_j +2(B_2)_{ij}dx_i\wedge dy_j +(B_3)_{ij}dy_i\wedge dy_j \Bigr) \label{B-field}
\end{equation}
called a B-field, and define an alternating matrix $B\in M(2n;\mathbb{R})$ by
\begin{equation*}
B=\left( \begin{array}{cc} B_1 & B_2 \\ -B_2^t & B_3 \end{array} \right).
\end{equation*}
Then, associated to the B-field (\ref{B-field}), we can consider the B-field transform of $\mathcal{I}_T$ whose representation matrix with respect to the basis $\mathcal{B}$ of $\Gamma(TX^n \oplus T^*X^n)$ is given by
\begin{equation*}
\left( \begin{array}{cc} I_{2n} & O \\ -B & I_{2n} \end{array} \right) \left( \begin{array}{cc} I_T & O \\ O & -I_T^t \end{array} \right) \left( \begin{array}{cc} I_{2n} & O \\ B & I_{2n} \end{array} \right),
\end{equation*}
where $I_{2n}$ denotes the identity matrix of order $2n$. 

Now, we roughly explain the correspondence between gerby deformations of $X^n$ and B-field transforms of $\mathcal{I}_T$. In our setting, although we can construct a gerbe in the sense of Hitchin-Chatterjee as a family of smooth complex line bundles on $O_i\cap O_j$ ($i$, $j\in I$) (cf. \cite{chat, h}), such smooth complex line bundles are all trivial since $O_i \cap O_j$ is contractible for each $i$, $j\in I$. This fact indicates that such a gerbe is trivial. However, we can consider a ``flat'' connection in the sense of Hitchin-Chatterjee, i.e., 0-connection, and 1-connection which is compatible with it over this trivial gerbe. In particular, 1-connection is a 2-form (not a 1-form), and it can be regarded as the above B-field (\ref{B-field}). 

By concerning the above, in our previous work \cite{gerby}, we discuss the case of ``$B_2$ type'', i.e., the case such that $B_2=\tau \in M(n;\mathbb{Z})$\footnote{Precisely speaking, we consider the case $B_2=\tau^t$ in \cite{gerby}.} and $B_1=B_3=O$. Actually, although we can consider the B-field transform of $\mathcal{I}_T$ in the case of ``$B_2$ type'' itself for $B_2=\tau \in M(n;\mathbb{R})$, the case $B_2=\tau \in M(n;\mathbb{Z})$ is particularly important in the sense that the corresponding deformation is closely related to a generalization of the $SL(2;\mathbb{Z})$-action (up to shifts) over $D^b(Coh(X^1))$ in the case of elliptic curves (cf. \cite{fm, slaction}) to the higher dimensional case (we omit the details of this fact here). Therefore, we focus on the case $B_2=\tau \in M(n;\mathbb{Z})$ only in \cite{gerby}. In this paper, as a sequel of our previous work \cite{gerby}, we investigate the B-field transform of $\mathcal{I}_T$ in the case of ``$B_1$ type'', i.e., the case such that $B_1=\tau \in M(n;\mathbb{R})$ and $B_2=B_3=O$, including the corresponding deformation in the symplectic geometry side, and discuss the homological mirror symmetry for the mirror pair of the deformed tori on the object level.

On the other hand, if we focus on the case ``$B_3$ type'', i.e., the case such that $B_3=\tau \in M(n;\mathbb{R})$ and $B_1=B_2=O$, then we need to consider the ``$\beta$-field'' transform depending on $\tau \in M(n;\mathbb{R})$ of the corresponding generalized complex structure in the symplectic geometry side. In general, it is closely related to a deformation quantization of the (complexified) symplectic torus $\check{X}^n$. This will be discussed in \cite{beta}.

\subsection{Related works}
As mentioned in subsection 1.2, although this paper is a sequel of our recent paper \cite{gerby}, these are based on some prior works which are explained below.

From the viewpoint of generalized complex geometry which is introduced by Hitchin and Gualtieri \cite{hitchin, Gual}, when we regard a complex manifold as a generalized complex manifold, as explained in subsection 1.2, it is thought that a B-field transform of the generalized complex structure is interpreted as a gerby deformation of the given complex manifold. This interpretation is mentioned in \cite{Gual}. Applications of generalized complex geometry to the mirror symmetry are particularly well studied by Oren Ben-Bassat in \cite{part1, part2}. There, more precisely, he discusses the homological mirror symmetry for generalized complex manifolds under ``adapted requirement'' (see \cite[Definition 4.2]{part1})\footnote{Note that the definition of the mirror transform used in \cite{part1, part2} differs from the definition employed in this paper.}. For example, in our setting, the generalized complex structure $\mathcal{I}_T$ induced from the complex structure on $X^n$ whose period matrix $T$ is given by $\mathbf{i}Y$ ($Y\in M(n;\mathbb{R})$ is positive definite) satisfies ``adapted requirement''. Moreover, the generalized complex structure induced from the canonical complex structure on a (general) complex manifold is also a typical example which satisfies ``adapted requirement''. In particular, relatively general B-field transforms of the generalized complex structure induced from the canonical complex structure on a complex manifold (B-field transforms treated in this paper are special cases of them) and the mirror dual counterparts are briefly explained in subsection 5.2 in \cite{part2} (under ``adapted requirement''). 

As examples of other related works, in \cite{kap1, kap2} (see also \cite{cal}), Kapustin and Orlov study the effects of B-fields on the bounded derived categories of coherent sheaves over abelian varieties (we do not assume that $X^n$ is an abelian variety in this paper) and the Fukaya categories over given symplectic manifolds, including the formulation of the statement of the ``deformed version'' of the homological mirror symmetry conjecture. In particular, although we do not treat superconformal vertex algebras in this paper, in \cite{kap1, kap2}, they examine properties of the bounded derived categories of twisted sheaves over the deformed abelian varieties from the viewpoint of superconformal vertex algebras.

On the other hand, in this paper, similarly as in our previous work \cite{gerby}, we intensively discuss certain B-field transforms of $\mathcal{I}_T$, i.e., gerby deformations of complex tori and the homological mirror symmetry without ``adapted requirement''. Although our discussions are basically based on the works explained in the above, we especially investigate the deformation of the category which is employed in the complex geometry side, i.e., the dg-category $DG_{X^n}$ associated to the deformation of $X^n$ by a certain flat gerbe and the counterpart in the symplectic geometry side in more detail than the existing works.

\subsection{Main results and the plan of this paper}
We take an arbitrary alternating matrix $\tau \in M(n;\mathbb{R})$, and fix it. In this paper, first, we define the gerby deformed $n$-dimensional complex torus $X_{\mathcal{G}_{\tau}}^n$ by regarding the B-field transform of $B_1$ type: $B_1=\tau \in M(n;\mathbb{R})$ for $\mathcal{I}_T$ as the twisting by the flat gerbe $\mathcal{G}_{\tau}$ in the sense of Hitchin-Chatterjee's work \cite{chat, h} (see also \cite{han}), where $\tau\in M(n;\mathbb{R})$ determines the flat connection (identified with the B-field of $B_1$ type essentially), and also consider its mirror partner $\check{X}_{\mathcal{G}_{\tau}}^n$ via generalized complex geometry. Next, we give the deformation of each object $E(s,\mathcal{L}^{\nabla})\to X^n$ as the twisted holomorphic line bundle with the twisted integrable connection $E(s,\mathcal{L}^{\nabla}(\tau))\to X_{\mathcal{G}_{\tau}}^n$ in the complex geometry side\footnote{Although it seems natural to write $E(s,\mathcal{L}^{\nabla})(\tau)$ instead of $E(s,\mathcal{L}^{\nabla}(\tau))$, we employ the notation $E(s,\mathcal{L}^{\nabla}(\tau))$ by concerning deformations in the symplectic geometry side.}. Moreover, we prove that these objects $E(s,\mathcal{L}^{\nabla}(\tau))$ forms a dg-category $DG_{X_{\mathcal{G}_{\tau}}^n}$ in Theorem \ref{maintheorem1}. In particular, we can regard this dg-category $DG_{X_{\mathcal{G}_{\tau}}^n}$ as a natural deformation of the dg-category $DG_{X^n}$ consisting of objects $E(s,\mathcal{L})$. On the other hand, the symplectic geometry side is a little complicated. More explicitly, although we can discuss the deformation of objects $(s,\mathcal{L}^{\nabla})$ according to the definition \cite[Definition 1.1]{Fuk} (see also \cite{kap1, kap2} and so on) in the case $\tau \in M(n;\mathbb{Z})$, it becomes necessary to reconsider how to define objects which should be treated on the deformed complexified symplectic torus under the assumption $\tau \not \in M(n;\mathbb{Z})$. Also, in the symplectic geometry side, note that each Lagrangian section $s$ is not deformed since the symplectic structure on $\check{X}^n$ is preserved associated to the deformation from $X^n$ to $X_{\mathcal{G}_{\tau}}^n$. By concerning the above problem, we decompose $\tau$ into its integer part $\tau_{\mathbb{Z}}\in M(n;\mathbb{Z})$ and its fractional part $\tau_{\rm frac}\in M(n;\mathbb{R})$:
\begin{equation*}
\tau =\tau_{\mathbb{Z}}+\tau_{\rm frac}.
\end{equation*}
Under this setting, we propose giving the deformation of each unitary local system $\mathcal{L}^{\nabla}$ as a ``twisted'' smooth complex line bundles with a ``twisted'' connection $\mathcal{L}^{\nabla}(\tau)$. Here, the twisting of $\mathcal{L}^{\nabla}$ is defined by using $\tau_{\rm frac}$, and at least, $\tau_{\mathbb{Z}}$ determines the first Chern character of $\mathcal{L}^{\nabla}(\tau)$ in the case $\tau \in M(n;\mathbb{Z})$, i.e., $\tau_{\rm frac}=O$. Then, we can check that these deformed objects $E(s,\mathcal{L}^{\nabla}(\tau))$ and $(s,\mathcal{L}^{\nabla}(\tau))$ are mirror dual to each other, and this result is given in Proposition \ref{generalizationmirror}. These imply that our deformations discussed in this paper seem to be natural from the viewpoint of the homological mirror symmetry for $(X_{\mathcal{G}_{\tau}}^n,\check{X}_{\mathcal{G}_{\tau}}^n)$.

This paper is organized as follows. In section 2, we recall the definition of gerbes in the sense of Hitchin-Chatterjee briefly. In section 3, we construct a mirror partner $\check{X}^n$ of a given $n$-dimensional complex torus $X^n$ via generalized complex geometry. In section 4, we define the deformation $X_{\mathcal{G}_{\tau}}^n$ of $X^n$ by $\mathcal{G}_{\tau}$ as the generalized complex torus whose generalized complex structure is given by the B-field transform of $B_1$ type: $B_1=\tau \in M(n;\mathbb{R})$ for $\mathcal{I}_T$, and construct its mirror partner $\check{X}_{\mathcal{G}_{\tau}}^n$. Main discussions are provided in section 5. The purpose of section 5 is to discuss the deformations on both sides associated to the deformation from $(X^n,\check{X}^n)$ to $(X_{\mathcal{G}_{\tau}}^n,\check{X}_{\mathcal{G}_{\tau}}^n)$. In subsection 5.1, we recall the correspondence between holomorphic line bundles with integrable connections $E(s,\mathcal{L}^{\nabla})$ and pairs $(s,\mathcal{L}^{\nabla})$, i.e., Lagrangian sections $s$ of $\check{X}^n\to \mathbb{R}^n/\mathbb{Z}^n$ with unitary local systems $\mathcal{L}^{\nabla}$ along them. In subsection 5.2, we discuss deformations $E(s,\mathcal{L}^{\nabla}(\tau))$ of objects $E(s,\mathcal{L}^{\nabla})$, and prove that these objects $E(s,\mathcal{L}^{\nabla}(\tau))$ naturally forms a dg-category $DG_{X_{\mathcal{G}_{\tau}}^n}$ in Theorem \ref{maintheorem1}. In subsection 5.3, we explain the problem which should be treated, and give the outline of an idea to solve the problem. In subsection 5.4, we describe the details of the idea which is mentioned in subsection 5.3, namely, we discuss deformations of objects $(s,\mathcal{L}^{\nabla})$ which are compatible with the counterpart in the mirror dual complex geometry side by focusing the division $\tau =\tau_{\mathbb{Z}}+\tau_{\rm frac}$. As a consequence, for given mirror dual objects $E(s,\mathcal{L}^{\nabla})$ and $(s,\mathcal{L}^{\nabla})$, we can confirm that the deformations of them are also mirror dual to each other. This result is given in Proposition \ref{generalizationmirror}. On the other hand, in general, the choice of a division of $\tau$ is not unique. For instance, when we focus on the trivial division $\tau =O+\tau$, we can also consider the twisting of $\mathcal{L}^{\nabla}$ associated to $\tau$. However, it can be expected that our proposal explained in this paper does not depend on the choice of such a division of $\tau$. We also describe this perspective in subsection 5.4.

\section{Gerbes in the sense of Hitchin-Chatterjee}
The purpose of this subsection is to recall the definition of gerbes\footnote[1]{Gerbes were originally introduced by Giraud in \cite{giraud}.} with connections in the sense of Hitchin-Chatterjee's work \cite{chat, h} briefly (see also \cite{han}). 

Roughly speaking, for a given smooth manifold $M$, a gerbe on $M$ is determined by an open covering $\{ U_i \}_{i\in I}$ of $M$ and an element of $H^2(M;\mathcal{A}^*)$, where $\mathcal{A}^*$ is the sheaf of nowhere vanishing smooth functions on $M$, so we need to focus on each intersection $U_{ijk}$ of three open sets $U_i$, $U_j$, $U_k$ if we consider gerbes according to the original definition (we use the notations $U_{ij}:=U_i\cap U_j$, $U_{ijk}:=U_i\cap U_j\cap U_k$, etc. in the above sense). On the other hand, Hitchin and Chatterjee propose a way of defining gerbes as families of line bundles defined on intersections $U_{ij}$ of two open sets. This definition is given as follows.
\begin{definition} \label{gerbe}
A gerbe $\mathcal{G}(I,L,\theta)$ on a smooth manifold $M$ is defined by the following data\textup{:} \\
\textup{(i)} An open covering $\{ U_i \}_{i\in I}$ of $M$. \\
\textup{(ii)} A family $L:=\{ L_{ij} \}_{i,j\in I}$ of smooth complex line bundles $L_{ij}\to U_{ij}$ such that $L_{ji}\cong L_{ij}^*$ on each $U_{ij}$ and $L_{ijk}:=L_{ij}|_{U_{ijk}}\otimes L_{jk}|_{U_{ijk}}\otimes L_{ki}|_{U_{ijk}}$ is isomorphic to $\mathcal{O}_{ijk}$ on each $U_{ijk}$, where $L_{ij}^*$ is the dual of $L_{ij}$ and $\mathcal{O}_{ijk}$ is the trivial complex line bundle on $U_{ijk}$. \\
\textup{(iii)} A family $\theta:=\{ \theta_{ijk} \}_{i,j,k\in I}$ of smooth sections $\theta_{ijk}\in \Gamma(U_{ijk};L_{ijk})$ such that $\theta_{ijk}=\theta_{jik}^{-1}=\theta_{ikj}^{-1}=\theta_{kji}^{-1}$ on each $U_{ijk}$ and 
\begin{equation*}
(\delta \theta)_{ijkl}=\theta_{jkl}\otimes \theta_{ikl}^{-1}\otimes \theta_{ijl}\otimes \theta_{ijk}^{-1}=1 
\end{equation*}
holds on each $U_{ijkl}$, where $\Gamma(U_{ijk};L_{ijk})$ denotes the space of smooth sections of $L_{ijk}$ on $U_{ijk}$ and $\delta$ is the derivative in the sense of $\check{C}$ech cohomology.
\end{definition}
\begin{rem}
In the condition \textup{(ii)} in Definition \ref{gerbe}, a trivialization $L_{ijk}\cong \mathcal{O}_{ijk}$ is not being fixed. Moreover, we see that
\begin{align*}
&(\delta^2 L)_{ijkl} \\
&=(L_{jk}\otimes L_{kl}\otimes L_{lj})\otimes(L_{ik}\otimes L_{kl}\otimes L_{li})^{-1}\otimes (L_{ij}\otimes L_{jl}\otimes L_{li})\otimes (L_{ij}\otimes L_{jk}\otimes L_{ki})^{-1} \\ &=\mathcal{O}_{ijkl}
\end{align*}
holds on each $U_{ijkl}$ since each smooth complex line bundle $L_{ij}$ satisfies the condition $L_{ji}\cong L_{ij}^*=L_{ij}^{-1}$, where $\mathcal{O}_{ijkl}$ is the trivial complex line bundle on $U_{ijkl}$. Hence, the condition \textup{(iii)} in Definition \ref{gerbe} states that $(\delta \theta)_{ijkl}$ is the canonical trivialization $1$ of $(\delta^2 L)_{ijkl}=\mathcal{O}_{ijkl}$.
\end{rem}
Connections over a gerbe $\mathcal{G}(I,L,\theta)$ are defined as follows.
\begin{definition} \label{0-conn}
Let $\mathcal{G}(I,L,\theta)$ be a gerbe on a smooth manifold $M$. We call a family $\nabla:=\{ \nabla_{ij} \}_{i,j\in I}$ of connections $\nabla_{ij}$ on $L_{ij}\to U_{ij}$ such that
\begin{equation*}
\nabla_{ijk}\theta_{ijk}=0
\end{equation*}
on each $U_{ijk}$ a 0-connection over $\mathcal{G}(I,L,\theta)$. Here, $\nabla_{ijk}$ denotes the connection on $L_{ijk}$ which is induced from the connections $\nabla_{ij}$, $\nabla_{jk}$, $\nabla_{ki}$\textup{:}
\begin{equation*}
\nabla_{ijk}:=\nabla_{ij}+\nabla_{jk}+\nabla_{ki}.
\end{equation*}
\end{definition}
In particular, note that each $\nabla_{ijk}$ in Definition \ref{0-conn} is a flat connection since $\nabla_{ijk}$ is a connection on $L_{ijk}$ which is isomorphic to $\mathcal{O}_{ijk}$. Therefore, the curvature form $\Omega_{ij}$ of each connection $\nabla_{ij}$ satisfies the relation 
\begin{equation*}
\Omega_{ij}+\Omega_{jk}+\Omega_{ki}=(\delta \Omega)_{ijk}=0,
\end{equation*}
so $\{ \Omega_{ij} \}_{i,j\in I}$ defines a representative in $H^1(M;\mathcal{A}^2)$, where $\mathcal{A}^2$ is the sheaf of smooth 2-forms on $M$. This implies the existence of $\beta_i\in \Gamma(U_i;\mathcal{A}^2)$ such that
\begin{equation*}
(\delta \beta)_{ij}=\Omega_{ij}
\end{equation*}
for each $\Omega_{ij}$ since $\mathcal{A}^2$ is a fine sheaf, i.e., $H^1(M;\mathcal{A}^2)\cong 0$.
\begin{definition} \label{1-conn}
We call the family $\beta:=\{ \beta_i \}_{i\in I}$ of 2-forms $\beta_i$ in the above a 1-connection over $\mathcal{G}(I,L,\theta)$ which is compatible with the 0-connection $\nabla$.
\end{definition}
On the other hand, the definition of a ``holomorphic'' gerbe on a complex manifold is also given by replacing ``smooth complex line bundles $L_{ij}\to U_{ij}$'' and ``smooth sections $\theta_{ijk}$'' with ``holomorphic line bundles $L_{ij}\to U_{ij}$'' and ``holomorphic sections $\theta_{ijk}$'', respectively in Definition \ref{gerbe}. In particular, we can naturally rewrite the notions of 0 and 1-connections in Definition \ref{0-conn} and Definition \ref{1-conn} to the ``holomorphic'' version (see section 5 in \cite{chat}).

\section{A mirror pair $(X^n, \check{X}^n)$}
Let $T$ be a complex matrix of order $n\in \mathbb{N}$ such that $\mathrm{Im}T$ is positive definite and $\mathrm{det}T\not=0$\footnote[1]{Although we do not need to assume the condition $\mathrm{det}T\not=0$ when we define an $n$-dimensional complex torus $X^n=\mathbb{C}^n/\mathbb{Z}^n\oplus T\mathbb{Z}^n$, in our setting described below, the mirror partner $\check{X}^n$ of $X^n$ does not exist if $\mathrm{det}T=0$. However, we can avoid this problem and discuss the homological mirror symmetry even if $\mathrm{det}T=0$ by modifying the definition of the mirror partner of $X^n$ and a class of objects which we treat. This fact is discussed in \cite{kazushi}.}. We denote the $n$-dimensional complex torus $\mathbb{C}^n/\mathbb{Z}^n\oplus T\mathbb{Z}^n$ by $X^n$:
\begin{equation*}
X^n:=\mathbb{C}^n/\mathbb{Z}^n\oplus T\mathbb{Z}^n.
\end{equation*}
In this section, we define a mirror partner $\check{X}^n$ of $X^n$ via the framework of generalized complex geometry (\cite{Gual, part1, part2} etc.). Throughout this section, for a smooth manifold $M$ and a smooth vector bundle $E$, $TM$, $T^*M$ and $\Gamma(E)$ denote the tangent bundle on $M$, the cotangent bundle on $M$ and the space of smooth sections of $E$, respectively.

First, let us prepare notations. We sometimes regard $X^n$ as a $2n$-dimensional real torus $\mathbb{R}^{2n}/\mathbb{Z}^{2n}$. We fix an $\epsilon>0$ small enough and let
\begin{align*}
O_{m_1\cdots m_n}^{l_1\cdots l_n}:=\biggl\{ \left( \begin{array}{ccc}x \\ y \end{array} \right)\in X^n \ \Bigl. \Bigr| \ &\frac{l_j-1}{3}-\epsilon <x_j <\frac{l_j}{3}+\epsilon, \\
&\frac{m_k-1}{3}-\epsilon <y_k <\frac{m_k}{3}+\epsilon, \ j,k=1,\cdots, n \biggr\}
\end{align*}
be subsets in $X^n$, where $l_j$, $m_k=1,2,3$, 
\begin{equation*}
x:=\left( \begin{array}{ccc} x_1 \\ \vdots \\ x_n \end{array} \right), \ y:=\left( \begin{array}{ccc} y_1 \\ \vdots \\ y_n \end{array} \right), 
\end{equation*}
and we identify $x_i\sim x_i+1$, $y_i\sim y_i+1$ for each $i=1$, $\cdots$, $n$. If necessary, we denote
\begin{equation*}
O_{m_1\cdots (m_k=m)\cdots m_n}^{l_1\cdots (l_j=l)\cdots l_n}
\end{equation*}
instead of $O_{m_1\cdots m_n}^{l_1\cdots l_n}$ in order to specify the values $l_j=l$, $m_k=m$. Then, $\{ O_{m_1\cdots m_n}^{l_1\cdots l_n} \}_{l_j, m_k=1,2,3}$ is an open covering of $X^n$, and we define the local coordinates of $O_{m_1\cdots m_n}^{l_1\cdots l_n}$ by 
\begin{equation*}
\left( \begin{array}{cccccc} x_1 \\ \vdots \\ x_n \\ y_1 \\ \vdots \\ y_n \end{array} \right)\in \mathbb{R}^{2n}.
\end{equation*}
Moreover, we locally express the complex coordinates 
\begin{equation*}
z:=\left( \begin{array}{ccc} z_1 \\ \vdots \\ z_n \end{array} \right)
\end{equation*}
of $X^n$ as $z=x+Ty$.

Below, we denote the transpose of a given matrix $A$ by $A^t$. In the above situation, we consider a linear map
\begin{equation*}
\mathcal{I}_T : \Gamma(TX^n\oplus T^*X^n)\rightarrow \Gamma(TX^n\oplus T^*X^n)
\end{equation*}
which is expressed locally as 
\begin{align}
&\mathcal{I}_T \left( \frac{\partial}{\partial x}^t, \frac{\partial}{\partial y}^t, dx^t, dy^t \right) \notag \\
&=\left( \frac{\partial}{\partial x}^t, \frac{\partial}{\partial y}^t, dx^t, dy^t \right) \notag \\
&\hspace{3.5mm} \left( \begin{array}{cccc} -XY^{-1} & -Y-XY^{-1}X & O & O \\ Y^{-1} & Y^{-1}X & O & O \\ O & O & (Y^{-1})^tX^t & -(Y^{-1})^t \\ O & O & Y^t+X^t(Y^{-1})^tX^t & -X^t(Y^{-1})^t \end{array} \right), \label{comp}
\end{align}
where
\begin{equation*}
X:=\mathrm{Re}T, \ Y:=\mathrm{Im}T,
\end{equation*}
and
\begin{equation*}
\frac{\partial}{\partial x}:=\left( \begin{array}{ccc} \frac{\partial}{\partial x_1} \\ \vdots \\ \frac{\partial}{\partial x_n} \end{array} \right), \ \frac{\partial}{\partial y}:=\left( \begin{array}{ccc} \frac{\partial}{\partial y_1} \\ \vdots \\ \frac{\partial}{\partial y_n} \end{array} \right), \
dx:=\left( \begin{array}{ccc} dx_1 \\ \vdots \\ dx_n \end{array} \right), \ dy:=\left( \begin{array}{ccc} dy_1 \\ \vdots \\ dy_n \end{array} \right).
\end{equation*}
We can check that this linear map $\mathcal{I}_T$ defines a generalized complex structure over $X^n$.

Let us define a mirror partner $\check{X}^n$ of the $n$-dimensional complex torus $X^n$. In general, for a $2n$-dimensional real torus $T^{2n}$ equipped with a generalized complex structure $\mathcal{I}$ over $T^{2n}$, its mirror dual generalized complex structure $\check{\mathcal{I}}$ over $T^{2n}$ is defined by
\begin{equation}
\check{\mathcal{I}}=\left( \begin{array}{cccc} I_n & O & O & O \\ O & O & O & I_n \\ O & O & I_n & O \\ O & I_n & O & O \end{array} \right) \mathcal{I} \left( \begin{array}{cccc} I_n & O & O & O \\ O & O & O & I_n \\ O & O & I_n & O \\ O & I_n & O & O \end{array} \right), \label{mirrordef}
\end{equation}
where $I_n$ is the identity matrix of order $n$ (cf. \cite{part1, part2, Kaj, kazushi}). According to the definition (\ref{mirrordef}), we can calculate the mirror dual generalized complex structure $\check{\mathcal{I}}_T$ of $\mathcal{I}_T$ locally as 
\begin{align}
&\check{\mathcal{I}}_T \left( \frac{\partial}{\partial \check{x}}^t, \frac{\partial}{\partial \check{y}}^t, d\check{x}^t, d\check{y}^t \right) \notag \\
&=\left( \frac{\partial}{\partial \check{x}}^t, \frac{\partial}{\partial \check{y}}^t, d\check{x}^t, d\check{y}^t \right) \notag \\
&\hspace{3.5mm} \left( \begin{array}{cccc} -XY^{-1} & O & O & -Y-XY^{-1}X \\ O & -X^t(Y^{-1})^t & Y^t+X^t(Y^{-1})^tX^t & O \\ O & -(Y^{-1})^t & (Y^{-1})^tX^t & O \\ Y^{-1} & O & O & Y^{-1}X \end{array} \right), \label{symp}
\end{align}
where 
\begin{equation*}
\check{x}:=\left( \begin{array}{ccc} x^1 \\ \vdots \\ x^n \end{array} \right), \ \check{y}:=\left( \begin{array}{ccc} y^1 \\ \vdots \\ y^n \end{array} \right) 
\end{equation*}
are the local coordinates of a $2n$-dimensional real torus $\mathbb{R}^{2n}/\mathbb{Z}^{2n}\approx (\mathbb{R}^n/\mathbb{Z}^n)\times (\mathbb{R}^n/\mathbb{Z}^n)$, and 
\begin{equation*}
\frac{\partial}{\partial \check{x}}:=\left( \begin{array}{ccc} \frac{\partial}{\partial x^1} \\ \vdots \\ \frac{\partial}{\partial x^n} \end{array} \right), \ \frac{\partial}{\partial \check{y}}:=\left( \begin{array}{ccc} \frac{\partial}{\partial y^1} \\ \vdots \\ \frac{\partial}{\partial y^n} \end{array} \right), \ d\check{x}:=\left( \begin{array}{ccc} dx^1 \\ \vdots \\ dx^n \end{array} \right), \ d\check{y}:=\left( \begin{array}{ccc} dy^1 \\ \vdots \\ dy^n \end{array} \right).
\end{equation*}
We can rewrite the representation matrix in (\ref{symp}) to
\begin{align*}
&\left( \begin{array}{cccc} I_n & O & O & O \\ O & I_n & O & O \\ O & -\mathrm{Re}(-(T^{-1})^t) & I_n & O \\ (\mathrm{Re}(-(T^{-1})^t))^t & O & O & I_n \end{array} \right) \\
&\left( \begin{array}{cccc} O & O & O & -((\mathrm{Im}(-(T^{-1})^t))^{-1})^t \\ O & O & (\mathrm{Im}(-(T^{-1})^t))^{-1} & O \\ O & -\mathrm{Im}(-(T^{-1})^t) & O & O \\ (\mathrm{Im}(-(T^{-1})^t))^t & O & O & O \end{array} \right) \\
&\left( \begin{array}{cccc} I_n & O & O & O \\ O & I_n & O & O \\ O & \mathrm{Re}(-(T^{-1})^t) & I_n & O \\ -(\mathrm{Re}(-(T^{-1})^t))^t & O & O & I_n \end{array} \right).
\end{align*}
In particular, the assumption $\mathrm{det}T\not=0$ implies $\mathrm{det}(Y+XY^{-1}X)\not=0$, and
\begin{align*}
&\mathrm{Re}(-(T^{-1})^t)=-((Y+XY^{-1}X)^{-1})^tX^t(Y^{-1})^t, \\ 
&\mathrm{Im}(-(T^{-1})^t)=((Y+XY^{-1}X)^{-1})^t.
\end{align*}
Thus, when we define a mirror partner of the $n$-dimensional complex torus $X^n$ according to the definition (\ref{mirrordef}), it is given by the $2n$-dimensional real torus $\mathbb{R}^{2n}/\mathbb{Z}^{2n}$ with the complexified symplectic form
\begin{equation*}
\tilde{\omega}^{\vee}:=2\pi \mathbf{i}d\check{x}^t (T^{-1})^t d\check{y}.
\end{equation*}
Here, $\mathbf{i}:=\sqrt{-1}$, and when we decompose $\tilde{\omega}^{\vee}$ into 
\begin{equation*}
\tilde{\omega}^{\vee}=2\pi d\check{x}^t \mathrm{Im}(-(T^{-1})^t) d\check{y}-\mathbf{i} \Bigl\{ 2\pi d\check{x}^t \mathrm{Re}(-(T^{-1})^t) d\check{y} \Bigr\},
\end{equation*}
its real part
\begin{equation*}
\omega^{\vee}:=2\pi d\check{x}^t \mathrm{Im}(-(T^{-1})^t) d\check{y} 
\end{equation*}
gives a symplectic form on $\mathbb{R}^{2n}/\mathbb{Z}^{2n}$, and the closed $2$-form
\begin{equation}
B^{\vee}:=2\pi d\check{x}^t \mathrm{Re}(-(T^{-1})^t) d\check{y} \label{bfield}
\end{equation}
is called a B-field. Moreover, we define two matrices $\omega_{\rm mat}^{\vee}$ and $B_{\rm mat}^{\vee}$ by
\begin{equation*}
\omega_{\rm mat}^{\vee}=\mathrm{Im}(-(T^{-1})^t)
\end{equation*}
and
\begin{equation*}
B_{\rm mat}^{\vee}=\mathrm{Re}(-(T^{-1})^t),
\end{equation*}
respectively:
\begin{equation*}
\omega^{\vee}=2\pi d\check{x}^t \omega_{\rm mat}^{\vee} d\check{y}, \ B^{\vee}=2\pi d\check{x}^t B_{\rm mat}^{\vee} d\check{y}.
\end{equation*}
Hereafter, we denote this complexified symplectic torus by
\begin{equation*}
\check{X}^n:=\Bigl( \mathbb{R}^{2n}/\mathbb{Z}^{2n}, \ \tilde{\omega}^{\vee}=2\pi \mathbf{i}d\check{x}^t (T^{-1})^t d\check{y} \Bigr).
\end{equation*}

\section{A gerby deformation $X_{\mathcal{G}_{\tau}^{\nabla}}^n$ of $X^n$ and its mirror partner $\check{X}_{\mathcal{G}_{\tau}^{\nabla}}^n$}
Below, we denote the set of real matrices of order $n$ by $M(n;\mathbb{R})$ (we use the notation $M(n;\mathbb{Z})$ in this sense). Let us take an arbitrary alternating matrix $\tau \in M(n;\mathbb{R})$: $\tau^t =-\tau$, and we fix it. In this section, we focus on the generalized complex torus whose generalized complex structure is given by the B-field transform of ``$B_1$ type: $B_1=\tau \in M(n;\mathbb{R})$ and $B_2=B_3=O$'' as mentioned in subsection 1.2 for the generalized complex structure $\mathcal{I}_T$ over $X^n$, and define the gerby deformed complex torus $X_{\mathcal{G}_{\tau}^{\nabla}}^n$ by regarding it as the deformation of $X^n$ by a flat gerbe $\mathcal{G}_{\tau}^{\nabla}$. Moreover, we also define a mirror partner $\check{X}_{\mathcal{G}_{\tau}^{\nabla}}^n$ of $X_{\mathcal{G}_{\tau}^{\nabla}}^n$ via the framework of generalized complex geometry. 

First, we prepare notations. For each open set $O_{m_1\cdots m_n}^{l_1\cdots l_n}$ in $X^n$, we put
\begin{equation*}
l:=(l_1\cdots l_n), \ m:=(m_1\cdots m_n),
\end{equation*}
and denote $O_{m_1\cdots m_n}^{l_1\cdots l_n}$ by $O_m^l$:
\begin{equation*}
O_m^l:=O_{m_1\cdots m_n}^{l_1\cdots l_n}.
\end{equation*}
The family $\{ O_m^l \}_{(l;m)\in I}$ of these open sets $O_m^l$ gives an open covering of $X^n$, i.e., 
\begin{equation*}
X^n=\bigcup_{(l;m)\in I}O_m^l,
\end{equation*}
where
\begin{equation*}
I:=\Bigl\{ (l;m)=(l_1\cdots l_n ; m_1\cdots m_n) \ | \ l_1,\cdots, l_n, m_1,\cdots, m_n=1,2,3 \Bigr\}.
\end{equation*}
Moreover, we denote
\begin{equation*}
O_{mm'}^{ll'}:=O_m^l \cap O_{m'}^{l'}, \ O_{mm'm''}^{ll'l''}:=O_m^l \cap O_{m'}^{l'} \cap O_{m''}^{l''}
\end{equation*}
for $(l;m)$, $(l';m')$, $(l'';m'')\in I$, and define a smooth complex line bundle $L_{mm'm''}^{ll'l''}\rightarrow O_{mm'm''}^{ll'l''}$ by
\begin{equation*}
L_{mm'm''}^{ll'l''}=\left. L_{mm'}^{ll'} \right|_{O_{mm'm''}^{ll'l''}}\otimes \left. L_{m'm''}^{l'l''} \right|_{O_{mm'm''}^{ll'l''}}\otimes \left.L_{m''m}^{l''l} \right|_{O_{mm'm''}^{ll'l''}}
\end{equation*}
for given three smooth complex line bundles $L_{mm'}^{ll'}\rightarrow O_{mm'}^{ll'}$, $L_{m'm''}^{l'l''}\rightarrow O_{m'm''}^{l'l''}$, $L_{m''m}^{l''l}\rightarrow O_{m''m}^{l''l}$.

Let us take a 2-form $^{\tau}\!B_m^l$ on the open set $O_m^l$ which is expressed locally as
\begin{equation}
^{\tau}\!B_m^l:=\pi dx^t \tau dx \label{btau}
\end{equation}
for each $(l;m)\in I$ by using $\tau \in M(n;\mathbb{R})$. Note that this $^{\tau}\!B_m^l$ is defined globally, and we call this globally defined 2-form $\{ ^{\tau}\!B_m^l \}_{(l;m)\in I}$ a B-field. We consider the B-field transform of the generalized complex structure $\mathcal{I}_T$ over $X^n$ (see also the local expression (\ref{comp})) associated to this B-field $\{ ^{\tau}\!B_m^l \}_{(l;m)\in I}$:
\begin{align*}
&\mathcal{I}_T\left( \{ ^{\tau}\!B_m^l \}_{(l;m)\in I} \right) \left( \frac{\partial}{\partial x}^t, \frac{\partial}{\partial y}^t, dx^t, dy^t \right) \\
&=\left( \frac{\partial}{\partial x}^t, \frac{\partial}{\partial y}^t, dx^t, dy^t \right) \left( \begin{array}{cccc} I_n & O & O & O \\ O & I_n & O & O \\ -\tau & O & I_n & O \\ O & O & O & I_n \end{array} \right) \mathcal{I}_T \left( \begin{array}{cccc} I_n & O & O & O \\ O & I_n & O & O \\ \tau & O & I_n & O \\ O & O & O & I_n \end{array} \right).
\end{align*}
Here, the matrix
\begin{equation*}
\left( \begin{array}{cc} -\tau & O \\ O & O \end{array} \right)
\end{equation*}
corresponds to the representation matrix of the map $\Gamma(TX^n)\to \Gamma(T^*X^n)$ associated to the B-field $\{ ^{\tau}\!B_m^l \}_{(l;m)\in I}$, i.e., $X\mapsto ^{\tau}\!B_m^l(X, \ \cdot \ )$ ($X\in \Gamma(TX^n)$) with respect to the bases 
\begin{equation*}
\left\{ \frac{\partial}{\partial x_1},\cdots, \frac{\partial}{\partial x_n}, \frac{\partial}{\partial y_1},\cdots, \frac{\partial}{\partial y_n} \right\}
\end{equation*}
of $\Gamma(TX^n)$ and 
\begin{equation*}
\Bigl\{ dx_1,\cdots, dx_n, dy_1,\cdots, dy_n \Bigr\}
\end{equation*}
of $\Gamma(T^*X^n)$. In particular, by direct calculations, we see that this B-field transform preserves the generalized complex structure over $X^n$, i.e., $\mathcal{I}_T(\{ ^{\tau}\!B_m^l \}_{(l;m)\in I})=\mathcal{I}_T$, if and only if $\tau =O$. 

We give the interpretation for the B-field transform $\mathcal{I}_T(\{ ^{\tau}\!B_m^l \}_{(l;m)\in I})$ of $\mathcal{I}_T$ from the viewpoint of the deformation of $X^n$ by a flat gerbe. It is clear that the open sets $O_m^l$ and their intersections are all contractible by the definition of the open covering $\{ O_m^l \}_{(l;m)\in I}$, so let us consider the trivial gerbe $\mathcal{G}(I,\mathcal{O},\xi)$ on $X^n$:
\begin{align*}
&\mathcal{O}:=\Bigl\{ \mathcal{O}_{mm'}^{ll'}\to O_{mm'}^{ll'} \Bigr\}_{(l;m),(l';m')\in I}, \\ &\xi:=\Bigl\{ \xi_{mm'm''}^{ll'l''}=1\in \Gamma(O_{mm'm''}^{ll'l''};\mathcal{O}_{mm'm''}^{ll'l''}) \Bigr\}_{(l;m),(l';m'),(l'';m'')\in I}.
\end{align*}
Now, we define a flat connection over $\mathcal{G}(I,\mathcal{O},\xi)$ by using the matrix $\tau \in M(n;\mathbb{R})$ which determines the B-field $\{ ^{\tau}\!B_m^l \}_{(l;m)\in I}$ as follows. For each $(l;m)$, $(l';m')\in I$, we take a 1-form $^{\tau}\!\omega_{mm'}^{ll'}$ on the open set $O_{mm'}^{ll'}$ which is expressed locally as
\begin{equation*}
^{\tau}\!\omega_{mm'}^{ll'}:=
\begin{cases}
\pi \tau_j^t dx & (l_j=1, l_j'=3, l_h=l_h' (h\not=j), m=m') \\ 
\pi \tau_j^t dx & (l_j=1, l_j'=3, l_h=l_h' (h\not=j), m_k=1, m_k'=3, m_i=m_i' (i\not=k)) \\ 
\pi \tau_j^t dx & (l_j=1, l_j'=3, l_h=l_h' (h\not=j), m_k=3, m_k'=1, m_i=m_i' (i\not=k)) \\ 
0 & (\mathrm{otherwise}),
\end{cases}
\end{equation*}
where $j$, $k=1$, $\cdots$, $n$, 
\begin{equation*}
\tau_j :=\left( \begin{array}{ccc} \tau_{j1} \\ \vdots \\ \tau_{jn} \end{array} \right)\in \mathbb{R}^n,
\end{equation*}
and assume the condition $^{\tau}\!\omega_{m'm}^{l'l}=-^{\tau}\!\omega_{mm'}^{ll'}$. By using them, we locally define a flat connection $^{\tau}\!\nabla_{mm'}^{ll'}$ on the trivial complex line bundle $\mathcal{O}_{mm'}^{ll'}\rightarrow O_{mm'}^{ll'}$ by
\begin{equation*}
^{\tau}\!\nabla_{mm'}^{ll'}=d-\mathbf{i}^{\tau}\!\omega_{mm'}^{ll'}
\end{equation*}
for each $(l;m)$, $(l';m')\in I$, where $d$ denotes the exterior derivative, and the family of $^{\tau}\!\nabla_{mm'}^{ll'}$ gives a 0-connection
\begin{equation*}
\nabla_{\tau}:=\Bigl\{ ^{\tau}\!\nabla_{mm'}^{ll'} \Bigr\}_{(l;m),(l';m')\in I}
\end{equation*}
over $\mathcal{G}(I,\mathcal{O},\xi)$. Then, we can easily check that the family of $-\mathbf{i}^{\tau}\!B_m^l$ defines a 1-connection
\begin{equation}
B_{\tau}:=\Bigl\{ -\mathbf{i}^{\tau}\!B_m^l \Bigr\}_{(l;m)\in I} \label{B}
\end{equation}
over $\mathcal{G}(I,\mathcal{O},\xi)$ which is compatible with the 0-connection $\nabla_{\tau}$. We denote the trivial gerbe $\mathcal{G}(I,\mathcal{O},\xi)$ with the flat connection $(\nabla_{\tau}, B_{\tau})$ by $\mathcal{G}_{\tau }^{\nabla}$:
\begin{equation*}
\mathcal{G}_{\tau }^{\nabla}:=\Bigl( \mathcal{G}(I,\mathcal{O},\xi), \ \nabla_{\tau}, \ B_{\tau} \Bigr).
\end{equation*}

Thus, we may identify the B-field $\{ ^{\tau}\!B_m^l \}_{(l;m)\in I}$ with the 1-connection $B_{\tau}$ over $\mathcal{G}(I, \mathcal{O}, \xi)$, so hereafter, we use the notation $\mathcal{I}_T(B_{\tau})$ instead of $\mathcal{I}_T(\{ ^{\tau}\!B_m^l \}_{(l;m)\in I})$. Concerning the above discussions, we regard the generalized complex torus whose generalized complex structure is given by $\mathcal{I}_T(B_{\tau})$ as the deformation of $X^n$ by the flat gerbe $\mathcal{G}_{\tau}^{\nabla}$, and denote it by
\begin{equation*}
X_{\mathcal{G}_{\tau}^{\nabla}}^n:=\Bigl( X^n, \ \mathcal{I}_T(B_{\tau}) \Bigr).
\end{equation*}

Let us calculate the mirror dual generalized complex structure $\check{\mathcal{I}}_T(B_{\tau})$ of $\mathcal{I}_T(B_{\tau})$ according to the definition (\ref{mirrordef}). By direct calculations, we see
\begin{align*}
&\check{\mathcal{I}}_T(B_{\tau}) \left( \frac{\partial}{\partial \check{x}}^t, \frac{\partial}{\partial \check{y}}^t, d\check{x}^t, d\check{y}^t \right) \\
&=\left( \frac{\partial}{\partial \check{x}}^t, \frac{\partial}{\partial \check{y}}^t, d\check{x}^t, d\check{y}^t \right) \left( \begin{array}{cccc} I_n & O & O & O \\ O & I_n & O & O \\ -\tau & O & I_n & O \\ O & O & O & I_n \end{array} \right) \check{\mathcal{I}}_T \left( \begin{array}{cccc} I_n & O & O & O \\ O & I_n & O & O \\ \tau & O & I_n & O \\ O & O & O & I_n \end{array} \right),
\end{align*}
and this fact implies that the mirror partner of $X_{\mathcal{G}_{\tau }^{\nabla}}^n$ is given by the $2n$-dimensional real torus $\mathbb{R}^{2n}/\mathbb{Z}^{2n}$ with the complexified symplectic form
\begin{equation*}
\tilde{\omega}_{\tau }^{\vee}:=2\pi \mathbf{i}d\check{x}^t (T^{-1})^t d\check{y}-\pi \mathbf{i}d\check{x}^t \tau d\check{x}.
\end{equation*}
Here, similarly as in the case of $\check{X}^n$, when we decompose $\tilde{\omega}_{\tau }^{\vee}$ into
\begin{align*}
\tilde{\omega}_{\tau }^{\vee}&=2\pi d\check{x}^t \omega_{\rm mat}^{\vee} d\check{y}-\mathbf{i} \Bigl( 2\pi d\check{x}^t B_{\rm mat}^{\vee} d\check{y}+\pi d\check{x}^t \tau d\check{x} \Bigr) \\
&=\omega^{\vee}-\mathbf{i}\Bigl( B^{\vee}+\pi d\check{x}^t \tau d\check{x} \Bigr),
\end{align*}
its real part $\omega^{\vee}$ gives a symplectic form on $\mathbb{R}^{2n}/\mathbb{Z}^{2n}$, and the closed 2-form
\begin{equation*}
B_{\tau }^{\vee}:=B^{\vee}+\pi d\check{x}^t \tau d\check{x} 
\end{equation*}
is a B-field. In particular, associated to the deformation of $X^n$ by $\mathcal{G}_{\tau }^{\nabla}$, in the symplectic geometry side, the symplectic form $\omega^{\vee}$ is preserved and the B-field $B^{\vee}$ on $\check{X}^n$ is twisted by $\pi d\check{x}^t \tau d\check{x}$ (see also the definition (\ref{bfield})). Hereafter, we denote this complexified symplectic torus (the mirror partner of $X_{\mathcal{G}_{\tau }}^n$ which is obtained by the definition (\ref{mirrordef})) by
\begin{equation*}
\check{X}_{\mathcal{G}_{\tau }^{\nabla}}^n:=\Bigl( \mathbb{R}^{2n}/\mathbb{Z}^{2n}, \ \tilde{\omega}_{\tau }^{\vee}=2\pi \mathbf{i}d\check{x}^t (T^{-1})^t d\check{y}-\pi \mathbf{i}d\check{x}^t \tau d\check{x} \Bigr).
\end{equation*}

\section{Deformations of objects and the homological mirror symmetry}
In general, the mirror pair $(X^n,\check{X}^n)$ can be regarded as the trivial special Lagrangian torus fibrations $X^n\rightarrow \mathbb{R}^n/\mathbb{Z}^n$ and $\check{X}^n\rightarrow \mathbb{R}^n/\mathbb{Z}^n$ on the same base space $\mathbb{R}^n/\mathbb{Z}^n$ by SYZ construction \cite{SYZ}. Then, we can construct a holomorphic line bundle with an integrable connection on $X^n$ from each pair of a Lagrangian section of $\check{X}^n\rightarrow \mathbb{R}^n/\mathbb{Z}^n$ and a unitary local system along it, and this transformation is called SYZ transform (see \cite{leung, A-P}). The purpose of this section is to discuss deformations of them over the deformed mirror pair $(X_{\mathcal{G}_{\tau}^{\nabla}}^n,\check{X}_{\mathcal{G}_{\tau}^{\nabla}}^n)$.

\subsection{Holomorphic line bundles with integrable connections and Lagrangian submanifolds with unitary local systems}
In this subsection, as preparations of main discussions, we recall the relation between Lagrangian submanifolds in $\check{X}^n$ with unitary local systems along them and the corresponding holomorphic line bundles with integrable connections on $X^n$ based on SYZ transform.

First, we explain the complex geometry side, namely, define a class of holomorphic line bundles $E_{(s,a,q)}\to X^n$ with integrable connections $\nabla_{(s,a,q)}$. Hereafter, we sometimes denote a holomorphic line bundle $E_{(s,a,q)}\to X^n$ with an integrable connection $\nabla_{(s,a,q)}$ by $E_{(s,a,q)}^{\nabla}$, i.e., $E_{(s,a,q)}^{\nabla}:=(E_{(s,a,q)}, \nabla_{(s,a,q)})$. In fact, we first construct it as a smooth complex line bundle on $X^n$ with a connection, and discuss the holomorphicity of such a smooth complex line bundle later (Proposition \ref{hol}). Before giving the strict definition of $E_{(s,a,q)}^{\nabla}$, we mention the idea of the construction of $E_{(s,a,q)}$ since the notations of transition functions of it are very complicated. Although we will give the details of the symplectic geometry side again later, in general, the Lagrangian submanifold in $\check{X}^n$ corresponding to a holomorphic line bundle with an integrable connection $E_{(s,a,q)}^{\nabla}$ is expressed locally as
\begin{equation*}
\left\{ \left( \begin{array}{cc} \check{x} \\ s(\check{x}) \end{array} \right)\in \check{X}^n \approx (\mathbb{R}^n/\mathbb{Z}^n)\times (\mathbb{R}^n/\mathbb{Z}^n) \right\},
\end{equation*}
where
\begin{equation*}
s(\check{x}):=\left( \begin{array}{ccc} s^1(\check{x}) \\ \vdots \\ s^n(\check{x}) \end{array} \right)
\end{equation*}
and locally defined smooth functions $s^1(\check{x}),\cdots ,s^n(\check{x})$ satisfy the relations
\begin{equation*}
\begin{array}{ccc} s^1\left( \begin{array}{cccc} x^1+1 \\ x^2 \\ \vdots \\ x^n \end{array} \right)=s^1\left( \begin{array}{cccc} x^1 \\ x^2 \\ \vdots \\ x^n \end{array} \right)+a_{11}, & \cdots & s^1\left( \begin{array}{cccc} x^1 \\ \vdots \\ x^{n-1} \\ x^n+1 \end{array} \right)=s^1\left( \begin{array}{cccc} x^1 \\ \vdots \\ x^{n-1} \\ x^n \end{array} \right)+a_{1n}, \\
\vdots & & \vdots \\
s^n\left( \begin{array}{cccc} x^1+1 \\ x^2 \\ \vdots \\ x^n \end{array} \right)=s^n\left( \begin{array}{cccc} x^1 \\ x^2 \\ \vdots \\ x^n \end{array} \right)+a_{n1}, & \cdots & s^n\left( \begin{array}{cccc} x^1 \\ \vdots \\ x^{n-1} \\ x^n+1 \end{array} \right)=s^n\left( \begin{array}{cccc} x^1 \\ \vdots \\ x^{n-1} \\ x^n \end{array} \right)+a_{nn} \end{array}
\end{equation*}
for a matrix
\begin{equation*}
a:=\left( \begin{array}{ccc} a_{11} & \cdots & a_{1n} \\ \vdots & \ddots & \vdots \\ a_{n1} & \cdots & a_{nn} \end{array} \right)\in M(n;\mathbb{Z}).
\end{equation*}
Then, the transition functions of $E_{(s,a,q)}$ are determined by this $a\in M(n;\mathbb{Z})$. Below, we give the strict definition of $E_{(s,a,q)}$. Let
\begin{equation*}
\psi ^{l_1 \cdots l_n}_{m_1 \cdots m_n} : O^{l_1 \cdots l_n}_{m_1 \cdots m_n}\to O^{l_1 \cdots l_n}_{m_1 \cdots m_n} \times \mathbb{C} \hspace{5mm} (l_j, m_k =1,2,3)
\end{equation*}
be a smooth section of $E_{(s,a,q)}|_{O^{l_1 \cdots l_n}_{m_1 \cdots m_n}}$. For each $j=1$, $\cdots$, $n$ and $k=1$, $\cdots$, $n$, the transition functions of $E_{(s,a,q)}$ are non-trivial on 
\begin{equation*}
O^{l_1 \cdots (l_j=3) \cdots l_n}_{m_1 \cdots m_n}\cap O^{l_1 \cdots (l_j=1) \cdots l_n}_{m_1 \cdots m_n}, \ O^{l_1 \cdots (l_j=3) \cdots l_n}_{m_1 \cdots (m_k=3) \cdots m_n}\cap O^{l_1 \cdots (l_j=1) \cdots l_n}_{m_1 \cdots (m_k=1) \cdots m_n}, \ O^{l_1 \cdots (l_j=3) \cdots l_n}_{m_1 \cdots (m_k=1) \cdots m_n}\cap O^{l_1 \cdots (l_j=1) \cdots l_n}_{m_1 \cdots (m_k=3) \cdots m_n},
\end{equation*}
and otherwise are trivial. For each $j=1$, $\cdots$, $n$ and $k=1$, $\cdots$, $n$, we define the transition functions on 
\begin{equation*}
O^{l_1 \cdots (l_j =3) \cdots l_n}_{m_1 \cdots m_n}\cap O^{l_1 \cdots (l_j =1) \cdots l_n}_{m_1 \cdots m_n}, \ O^{l_1 \cdots (l_j=3) \cdots l_n}_{m_1 \cdots (m_k=3) \cdots m_n}\cap O^{l_1 \cdots (l_j=1) \cdots l_n}_{m_1 \cdots (m_k=1) \cdots m_n}
\end{equation*}
and 
\begin{equation*}
O^{l_1 \cdots (l_j=3) \cdots l_n}_{m_1 \cdots (m_k=1) \cdots m_n}\cap O^{l_1 \cdots (l_j=1) \cdots l_n}_{m_1 \cdots (m_k=3) \cdots m_n}
\end{equation*}
by
\begin{align*}
&\biggl.\psi ^{l_1 \cdots (l_j =3) \cdots l_n}_{m_1 \cdots m_n} \biggr|_{O^{l_1 \cdots (l_j =3) \cdots l_n}_{m_1 \cdots m_n}\cap O^{l_1 \cdots (l_j =1) \cdots l_n}_{m_1 \cdots m_n}} \\
&=\mathrm{exp}\left( 2\pi \mathbf{i}a_j^t y \right) \biggl.\psi ^{l_1 \cdots (l_j =1) \cdots l_n}_{m_1 \cdots m_n} \biggr|_{O^{l_1 \cdots (l_j =3) \cdots l_n}_{m_1 \cdots m_n}\cap O^{l_1 \cdots (l_j =1) \cdots l_n}_{m_1 \cdots m_n}}, \\
&\biggl.\psi ^{l_1 \cdots (l_j =3) \cdots l_n}_{m_1 \cdots (m_k=3) \cdots m_n} \biggr|_{O^{l_1 \cdots (l_j =3) \cdots l_n}_{m_1 \cdots (m_k=3) \cdots m_n}\cap O^{l_1 \cdots (l_j =1) \cdots l_n}_{m_1 \cdots (m_k=1) \cdots m_n}} \\
&=\mathrm{exp}\left( 2\pi \mathbf{i}a_j^t y \right) \biggl.\psi ^{l_1 \cdots (l_j =1) \cdots l_n}_{m_1 \cdots (m_k=1) \cdots m_n} \biggr|_{O^{l_1 \cdots (l_j =3) \cdots l_n}_{m_1 \cdots (m_k=3) \cdots m_n}\cap O^{l_1 \cdots (l_j =1) \cdots l_n}_{m_1 \cdots (m_k=1) \cdots m_n}}
\end{align*}
and
\begin{align*}
&\biggl.\psi ^{l_1 \cdots (l_j =3) \cdots l_n}_{m_1 \cdots (m_k=1) \cdots m_n} \biggr|_{O^{l_1 \cdots (l_j =3) \cdots l_n}_{m_1 \cdots (m_k=1) \cdots m_n}\cap O^{l_1 \cdots (l_j =1) \cdots l_n}_{m_1 \cdots (m_k=3) \cdots m_n}} \\
&=\mathrm{exp}\left( 2\pi \mathbf{i}a_j^t y \right) \biggl.\psi ^{l_1 \cdots (l_j =1) \cdots l_n}_{m_1 \cdots (m_k=3) \cdots m_n} \biggr|_{O^{l_1 \cdots (l_j =3) \cdots l_n}_{m_1 \cdots (m_k=1) \cdots m_n}\cap O^{l_1 \cdots (l_j =1) \cdots l_n}_{m_1 \cdots (m_k=3) \cdots m_n}},
\end{align*}
respectively, where 
\begin{equation*}
a_j:=\left( \begin{array}{ccc} a_{1j} \\ \vdots \\ a_{nj} \end{array} \right)\in \mathbb{Z}^n. 
\end{equation*}
In particular, it is easy to check that the cocycle condition is satisfied. Moreover, by using a locally defined smooth function
\begin{equation*}
s(x):=\left( \begin{array}{ccc} s^1(x) \\ \vdots \\ s^n(x) \end{array} \right)
\end{equation*}
satisfying the relations
\begin{equation}
s\left( \begin{array}{cccc} x_1+1 \\ x_2 \\ \vdots \\ x_n \end{array} \right)=s\left( \begin{array}{cccc} x_1 \\ x_2 \\ \vdots \\ x_n \end{array} \right)+a_1, \ \cdots, \ s\left( \begin{array}{cccc} x_1 \\ \vdots \\ x_{n-1} \\ x_n+1 \end{array} \right)=s\left( \begin{array}{cccc} x_1 \\ \vdots \\ x_{n-1} \\ x_n \end{array} \right)+a_n \label{section}
\end{equation}
and a constant vector
\begin{equation*}
q:=\left( \begin{array}{ccc} q_1 \\ \vdots \\ q_n \end{array} \right)\in \mathbb{R}^n,
\end{equation*}
we locally define a connection $\nabla_{(s,a,q)}$ on $E_{(s,a,q)}$ by
\begin{equation}
\nabla_{(s,a,q)}=d+\omega_{(s,a,q)}, \ \ \ \omega_{(s,a,q)}=-2\pi \mathbf{i} \Bigl( s(x)^t+q^t T \Bigr)dy. \label{connection}
\end{equation}
Also, for the constant vectors $a_j\in \mathbb{Z}^n$ ($j=1$, $\cdots$, $n$), we set
\begin{equation*}
a:=(a_1,\cdots, a_n)\in M(n;\mathbb{Z}).
\end{equation*}
We can easily check that the above $\nabla_{(s,a,q)}$ is compatible with the transition functions of $E_{(s,a,q)}$ since the function $s$ satisfies the relations (\ref{section}). 

Here, we recall the definition of the integrability of connections following \cite{steven}. Let $X$ be a compact K$\ddot{\mathrm{a}}$hler manifold, and we consider a smooth complex vector bundle $E$ with a connection $D$ on $X$. In particular, $D^{(0,1)}$ denotes the (0,1)-part of $D$. 
\begin{definition}[{\cite[Definition 1.4.1]{steven}}] \label{i_c}
A connection $D$ is called integrable if $(D^{(0,1)})^2=0$.
\end{definition}
Therefore, $(E,D)$ can be considered as a holomorphic vector bundle on $X$ if $D$ is integrable.
\begin{rem} \label{i_c_r}
In this paper, as an analogue of Definition \ref{i_c}, for a twisted connection on an $\alpha$-twisted smooth complex vector bundle \textup{(}this $\alpha$ denotes the 2-cocycle which defines a given gerbe\textup{)}, we call the twisted connection a twisted integrable connection on the $\alpha$-twisted smooth complex vector bundle if the \textup{(0,2)}-part of its curvature form vanishes.
\end{rem}
Now, we give the following proposition without its proof since it can be proved similarly as in \cite[Proposition 3.1]{bijection} (see also \cite[Proposition 4.3]{kazushi}, \cite[Proposition 2.1]{exact} and \cite[Proposition 4.1]{gerby}), where
\begin{equation*}
\frac{\partial s}{\partial x}(x):=\left( \begin{array}{ccc} \frac{\partial s^1}{\partial x_1}(x) & \cdots & \frac{\partial s^1}{\partial x_n}(x) \\ \vdots & \ddots & \vdots \\ \frac{\partial s^n}{\partial x_1}(x) & \cdots & \frac{\partial s^n}{\partial x_n}(x) \end{array} \right).
\end{equation*}
\begin{proposition} \label{hol}
For a given locally defined smooth function $s$ satisfying the relations \textup{(\ref{section})} and a given constant vector $q\in \mathbb{R}^n$, the connection $\nabla_{(s,a,q)}$ is an integrable connection on $E_{(s,a,q)}\to X^n$ if and only if 
\begin{equation*}
\frac{\partial s}{\partial x}(x)T=\left( \frac{\partial s}{\partial x}(x)T \right)^t
\end{equation*}
holds.
\end{proposition}
Although we will explain the details in subsection 5.2 (see Theorem \ref{maintheorem1}), these holomorphic line bundles with integrable connections $E_{(s,a,q)}^{\nabla}$ forms a dg-category $DG_{X^n}$ (see also \cite{kajiura, exact} for instance). In general, for any $A_{\infty}$-category $\mathscr{A}$, we can construct the enhanced triangulated category $Tr(\mathscr{A})$ by using Bondal-Kapranov-Kontsevich construction \cite{bondal, Kon}. We expect that the dg-category $DG_{X^n}$ generates the bounded derived category of coherent sheaves $D^b(Coh(X^n))$ over $X^n$ in the sense of Bondal-Kapranov-Kontsevich construction:
\begin{equation*}
Tr(DG_{X^n})\cong D^b(Coh(X^n)).
\end{equation*}
At least, it is known that it split generates $D^b(Coh(X^n))$ when $X^n$ is an abelian variety (cf. \cite{orlov, abouzaid}).

Next, we explain the symplectic geometry side, namely, give the objects in the symplectic geometry side corresponding to holomorphic line bundles with integrable connections $E_{(s,a,q)}^{\nabla}$. 

Before starting main discussions, we recall the definition of objects of the Fukaya category over a given complexified symplectic manifold (see \cite[Definition 1.1]{Fuk} for example). Let $M_{\omega, B}:=(M, 2\pi \omega-2\pi \mathbf{i}B)$ be a complexified symplectic manifold, where $M$ is a smooth even-dimensional real manifold, $2\pi \omega$ is a symplectic form on $M$, and $2\pi B$ is a B-field on $M$. For this $M_{\omega, B}$, an object of the Fukaya category over $M_{\omega, B}$ consists of a Lagrangian submanifold $L$ in $M_{\omega, B}$, i.e., a submanifold $L$ in $M_{\omega, B}$ which satisfies
\begin{equation*}
\mathrm{dim}_{\mathbb{R}}L=\frac{1}{2}\mathrm{dim}_{\mathbb{R}}M, \ \Bigl. 2\pi \omega \Bigr|_L=0,
\end{equation*}
and a smooth complex line bundle $\mathcal{L}\to L$ with a unitary connection $\nabla_{\mathcal{L}}$ such that 
\begin{equation*}
\Omega_{\nabla_{\mathcal{L}}}=\Bigl. -2\pi \mathbf{i}B \Bigr|_L,
\end{equation*}
where $\Omega_{\nabla_{\mathcal{L}}}$ denotes the curvature form of $\nabla_{\mathcal{L}}$.

Let us consider the $n$-dimensional submanifold
\begin{equation*}
L_{(s,a)}:=\left\{ \left( \begin{array}{cc} \check{x} \\ s(\check{x}) \end{array} \right)\in \check{X}^n\approx (\mathbb{R}^n/\mathbb{Z}^n)\times (\mathbb{R}^n/\mathbb{Z}^n) \right\}
\end{equation*}
in $\check{X}^n$ by using a locally defined smooth function
\begin{equation*}
s(\check{x}):=\left( \begin{array}{ccc} s^1(\check{x}) \\ \vdots \\ s^n(\check{x}) \end{array} \right)
\end{equation*}
satisfying the relations
\begin{equation}
s\left( \begin{array}{cccc} x^1+1 \\ x^2 \\ \vdots \\ x^n \end{array} \right)=s\left( \begin{array}{cccc} x^1 \\ x^2 \\ \vdots \\ x^n \end{array} \right)+a_1, \ \cdots, \ s\left( \begin{array}{cccc} x^1 \\ \vdots \\ x^{n-1} \\ x^n+1 \end{array} \right)=s\left( \begin{array}{cccc} x^1 \\ \vdots \\ x^{n-1} \\ x^n \end{array} \right)+a_n. \label{lagsection}
\end{equation}
As explained in the above, we use the notation $a$ in the sense
\begin{equation*}
a:=(a_1,\cdots, a_n)\in M(n;\mathbb{Z}),
\end{equation*}
and for later convenience, we set
\begin{equation*}
\frac{\partial s}{\partial \check{x}}(\check{x}):=\left( \begin{array}{ccc} \frac{\partial s^1}{\partial x^1}(\check{x}) & \cdots & \frac{\partial s^1}{\partial x^n}(\check{x}) \\ \vdots & \ddots & \vdots \\ \frac{\partial s^n}{\partial x^1}(\check{x}) & \cdots & \frac{\partial s^n}{\partial x^n}(\check{x}) \end{array} \right).
\end{equation*}
We further take the trivial complex line bundle $\mathcal{L}_{(s,a,q)}:=\mathcal{O}_{L_{(s,a)}}\to L_{(s,a)}$ with the flat connection
\begin{equation*}
\nabla_{\mathcal{L}_{(s,a,q)}}:=d-2\pi \mathbf{i} q^t d\check{x} 
\end{equation*}
which is defined by a constant vector
\begin{equation*}
q:=\left( \begin{array}{ccc} q_1 \\ \vdots \\ q_n \end{array} \right)\in \mathbb{R}^n.
\end{equation*}
Hereafter, we set $\mathcal{L}_{(s,a,q)}^{\nabla}:=(\mathcal{L}_{(s,a,q)}, \nabla_{\mathcal{L}_{(s,a,q)}})$, and let us denote a pair of an $n$-dimensional submanifold $L_{(s,a)}$ in $\check{X}^n$ and the trivial complex line bundle with a flat connection $\mathcal{L}_{(s,a,q)}^{\nabla}$ by $\mathscr{L}_{(s,a,q)}^{\nabla}$: $\mathscr{L}_{(s,a,q)}^{\nabla}:=(L_{(s,a)}, \mathcal{L}_{(s,a,q)}^{\nabla})$. Then, the following proposition holds. Here, we omit its proof since it can be proved in a similar way which is written in subsection 4.1 in \cite{bijection} (see also section 4 in \cite{kazushi}, section 2 in \cite{exact} and \cite[Proposition 4.2]{gerby}).
\begin{proposition} \label{fukob}
For a given locally defined smooth function $s$ satisfying the relations \textup{(\ref{lagsection})} and a given constant vector $q\in \mathbb{R}^n$, the pair $\mathscr{L}_{(s,a,q)}^{\nabla}$ gives an object of the Fukaya category over $\check{X}^n$ if and only if
\begin{equation*}
\frac{\partial s}{\partial \check{x}}(\check{x})T=\left( \frac{\partial s}{\partial \check{x}}(\check{x})T \right)^t
\end{equation*}
holds.
\end{proposition}
We denote the full subcategory of the Fukaya category over $\check{X}^n$ consisting of objects $\mathscr{L}_{(s,a,q)}^{\nabla}$ satisfying the condition $\frac{\partial s}{\partial \check{x}}(\check{x})T=\left( \frac{\partial s}{\partial \check{x}}(\check{x})T \right)^t$ by $Fuk_{\rm sub}(\check{X}^n)$.

We give an interpretation for the above discussions from the viewpoint of SYZ construction. We can regard the complexified symplectic torus $\check{X}^n$ as the trivial special Lagrangian torus fibration $\check{X}^n\rightarrow \mathbb{R}^n/\mathbb{Z}^n$, where $\check{x}$ is the local coordinate system of the base space $\mathbb{R}^n/\mathbb{Z}^n$ and $\check{y}$ is the local coordinate system of the fiber of $\check{X}^n\rightarrow \mathbb{R}^n/\mathbb{Z}^n$. Then, we can regard each affine Lagrangian submanifold $L_{(s,a)}$ in $\check{X}^n$ as the graph of the section $s : \mathbb{R}^n/\mathbb{Z}^n\rightarrow \check{X}^n$ of $\check{X}^n\rightarrow \mathbb{R}^n/\mathbb{Z}^n$.

Finally, by comparing Proposition \ref{hol} with Proposition \ref{fukob}, we immediately obtain the following proposition. In particular, this indicates that 
\begin{equation*}
E_{(s,a,q)}^{\nabla}\in DG_{X^n}
\end{equation*}
and 
\begin{equation*}
\mathscr{L}_{(s,a,q)}^{\nabla}\in Fuk_{\rm sub}(\check{X})
\end{equation*}
are mirror dual to each other.
\begin{proposition} \label{mirror}
For a given locally defined smooth function $s$ satisfying the relations \textup{(\ref{section})}, \textup{(\ref{lagsection})} and a given constant vector $q\in \mathbb{R}^n$, the connection $\nabla_{(s,a,q)}$ is an integrable connection on $E_{(s,a,q)}\to X^n$ if and only if the pair $\mathscr{L}_{(s,a,q)}^{\nabla}$ gives an object of the Fukaya category over $\check{X}^n$.
\end{proposition}

\subsection{A deformation $DG_{X_{\mathcal{G}_{\tau}^{\nabla}}^n}$ of $DG_{X^n}$}
In this subsection, we consider the deformation of holomorphic line bundles with integrable connections $E_{(s,a,q)}^{\nabla}$ associated to the deformation from $X^n$ to $X_{\mathcal{G}_{\tau}^{\nabla}}^n$. In particular, the main purpose of this subsection is to prove that such deformed objects naturally forms a dg-category, and this result is given in Theorem \ref{maintheorem1}.

First, as a preparations of main discussions, we quickly recall twisted line bundles with twisted connections. In general, when we consider the deformation $M_{\mathcal{G}}$ of a given complex manifold $M$ by a gerbe $\mathcal{G}$ with 0 and 1-connections, each smooth complex line bundle $L\to M$ with a connection $\nabla$ is deformed to the twisted smooth complex line bundle $L^{\mathcal{G}}\to M_{\mathcal{G}}$ with the twisted connection $\nabla^{\mathcal{G}}$. Below, we explain this fact more precisely. Let $M$ be a complex manifold and we take an open covering $\{ U_i \}_{i\in I}$ of $M$. For each smooth complex line bundle $L\to M$ with a connection $\{ \nabla_i \}_{i\in I}$ which is defined by a family of transition functions $\{ \varphi_{ij} \}_{i,j\in I}$, $\nabla_i$ and $\varphi_{ij}$ satisfy the relation
\begin{equation*}
\nabla_j-\varphi_{ij}^{-1}\nabla_i \varphi_{ij}=0
\end{equation*}
on $U_i\cap U_j$ for each $i,j\in I$. Here, we take the trivial gerbe $\mathcal{G}$ with a non-trivial 0-connection $\{ \nabla_{ij}=d+\omega_{ij} \}_{i,j\in I}$, and consider the deformation $M_{\mathcal{G}}$ of $M$ by this gerbe $\mathcal{G}$ (we only consider the case of the trivial gerbe with a non-trivial 0-connection since $\mathcal{G}_{\tau }^{\nabla}$ is trivial as a gerbe). Then, the family of transition functions $\{ \varphi_{ij} \}_{i,j\in I}$ is preserved under this deformation since $\mathcal{G}$ is trivial. This result indicates that the deformation $L^{\mathcal{G}}\to M_{\mathcal{G}}$ (as a bundle) of $L\to M$ by $\mathcal{G}$ is defined by the same family of transition functions $\{ \varphi_{ij}^{\mathcal{G}}:=\varphi_{ij} \}_{i,j\in I}$. In other words, $L^{\mathcal{G}}$ is an $id$-twisted smooth complex line bundle on $M_{\mathcal{G}}$ in the sense of \cite[Definition 2.1]{k-htwisted}\footnote{Although the twisting of $L$ by $\mathcal{G}$ is actually trivial since $\mathcal{G}$ is trivial, we regard the deformation $L^{\mathcal{G}}$ of $L$ as an $id$-twisted line bundle in order to specify that the deformation of $L$ is caused by the flat gerbe $(\mathcal{G}, \{ \nabla_{ij} \}_{i,j\in I})$.}. On the other hand, the twisted connection $\{ \nabla_i^{\mathcal{G}} \}_{i\in I}$ on the deformed object $L^{\mathcal{G}}\to M_{\mathcal{G}}$ is defined by the relation
\begin{equation}
\nabla_j^{\mathcal{G}}-\left( \varphi_{ij}^{\mathcal{G}} \right)^{-1} \nabla_i^{\mathcal{G}} \left( \varphi_{ij}^{\mathcal{G}} \right)=\omega_{ij} \label{twconn}
\end{equation}
on $U_i\cap U_j$ for each $i,j\in I$. We need to be careful when we consider the curvature form of the twisted connections $\{ \nabla_i^{\mathcal{G}} \}_{i\in I}$. Let $\{ \omega_i^{\mathcal{G}} \}_{i\in I}$ be a connection 1-form of $\{ \nabla_i^{\mathcal{G}} \}_{i\in I}$, and for each $i\in I$, we define the ``local'' curvature form\footnote{This word is employed in section 7 in \cite{k-theory}.} $\tilde{\Omega}_i^{\mathcal{G}}$ of $\{ \nabla_i^{\mathcal{G}} \}_{i\in I}$ by
\begin{equation*}
\tilde{\Omega}_i^{\mathcal{G}}=d\omega_i^{\mathcal{G}}+\omega_i^{\mathcal{G}}\wedge \omega_i^{\mathcal{G}}=d\omega_i^{\mathcal{G}}.
\end{equation*}
Unfortunately, on each intersection $U_i\cap U_j$, this ``local'' curvature form $\{ \tilde{\Omega}_i^{\mathcal{G}} \}_{i\in I}$ does not agree, namely, the relation
\begin{equation*}
\tilde{\Omega}_j^{\mathcal{G}}=\left( \varphi_{ij}^{\mathcal{G}} \right)^{-1} \tilde{\Omega}_i^{\mathcal{G}} \left( \varphi_{ij}^{\mathcal{G}} \right)+d\omega_{ij}
\end{equation*}
holds. However, when we set
\begin{equation*}
\Omega_i^{\mathcal{G}}:=\tilde{\Omega}_i^{\mathcal{G}}-B_i
\end{equation*}
on each $U_i$ by using the 1-connection $\{ B_i \}_{i\in I}$ over $\mathcal{G}$ which is compatible with the 0-connection $\{ \nabla_i^{\mathcal{G}} \}_{i\in I}$, there exists a unique 2-form $\Omega^{\mathcal{G}}$ on $M_{\mathcal{G}}$ with values in $\mathrm{End}(L^{\mathcal{G}})$ such that $\Omega^{\mathcal{G}}|_{U_i}=\Omega_i^{\mathcal{G}}$ for each $i\in I$. This 2-form $\Omega^{\mathcal{G}}$ is called the curvature form of the twisted connection $\{ \nabla_i^{\mathcal{G}} \}_{i\in I}$ (see subsection 2.3 and subsection 2.4 in \cite{k-htwisted}).

Let us discuss the deformation of smooth complex line bundles with connections $E_{(s,a,q)}^{\nabla}\to X^n$, namely, we define the $id$-twisted smooth complex line bundle $E_{(s,a,q)}^{\tau}\to X_{\mathcal{G}_{\tau }^{\nabla}}^n$ with the twisted connection $\nabla_{(s,a,q)}^{\tau}$ as the deformation of a given smooth complex line bundle with a connection $E_{(s,a,q)}^{\nabla}\to X^n$ by $\mathcal{G}_{\tau }^{\nabla}$. The transition functions of $E_{(s,a,q)}^{\tau}$ coincide with the transition functions of $E_{(s,a,q)}$ since $\mathcal{G}(I,\mathcal{O},\xi)$ is trivial. On the other hand, the connection $\nabla_{(s,a,q)}$ on $E_{(s,a,q)}$ is locally deformed to the following twisted connection $\nabla_{(s,a,q)}^{\tau}$ (see also the definition (\ref{connection})):
\begin{align*}
&\nabla_{(s,a,q)}^{\tau}=d+\omega_{(s,a,q)}^{\tau}, \\
&\omega_{(s,a,q)}^{\tau}:=\omega_{(s,a,q)}-\pi \mathbf{i}x^t \tau dx=-2\pi \mathbf{i} \Bigl( s(x)^t+q^tT \Bigr)dy-\pi \mathbf{i}x^t\tau dx. 
\end{align*}
In fact, we can easily check that this $\nabla_{(s,a,q)}^{\tau}$ is compatible with the transition functions of $E_{(s,a,q)}^{\tau}$, i.e., the relations of the form (\ref{twconn}) hold. Similarly as in the case of $E_{(s,a,q)}^{\nabla}$, we sometimes denote this deformed object by
\begin{equation*}
E_{(s,a,q)}^{\nabla}(\tau):=\Bigl( E_{(s,a,q)}^{\tau}, \ \nabla_{(s,a,q)}^{\tau} \Bigr). 
\end{equation*}
Although the transition functions are preserved under this deformation, we see that the connection 1-form $\omega_{(s,a,q)}$ is twisted by the 1-form
\begin{equation}
-\pi \mathbf{i}x^t \tau dx \label{1form}
\end{equation}
under this deformation (see also the definition (\ref{connection})). We can further obtain the 2-form
\begin{equation}
-\pi \mathbf{i}dx^t \tau dx \label{2form}
\end{equation}
from the 1-form (\ref{1form}) by considering its derivative, and this 2-form (\ref{2form}) is the 1-connection over $\mathcal{G}(I,\mathcal{O},\xi)$ itself which is compatible with the 0-connection $\nabla_{\tau}$ (see the local expression (\ref{btau}) and the definition (\ref{B})). The 2-form 
\begin{equation*}
\pi dx^t \tau dx
\end{equation*}
in the local expression (\ref{2form}) is also called a B-field as mentioned in section 4. Here, we discuss the holomorphicity of $E_{(s,a,q)}^{\nabla}(\tau)$. We see that the following proposition holds.
\begin{proposition} \label{twhol}
For a given locally defined smooth function $s$ satisfying the relation \textup{(\ref{section})} and a given constant vector $q\in \mathbb{R}^n$, the twisted connection $\nabla_{(s,a,q)}^{\tau}$ is a twisted integrable connection on $E_{(s,a,q)}^{\tau}\rightarrow X_{\mathcal{G}_{\tau }^{\nabla}}^n$ if and only if
\begin{equation}
\frac{\partial s}{\partial x}(x)T=\left( \frac{\partial s}{\partial x}(x)T \right)^t \label{symm_comp}
\end{equation}
holds, and then, the \textup{(0,2)}-part of the ``local'' curvature form $\tilde{\Omega}_{(s,a,q)}^{\tau}$ is expressed locally as
\begin{equation*}
\left( \tilde{\Omega}_{(s,a,q)}^{\tau} \right)^{(0,2)}=-\pi \mathbf{i} d\bar{z}^t ((T-\bar{T})^{-1})^t T^t \tau T(T-\bar{T})^{-1} d\bar{z}.
\end{equation*}
\end{proposition}
\begin{proof}
It is easy to check that the ``local'' curvature form $\tilde{\Omega}_{(s,a,q)}^{\tau}$ and the curvature form $\Omega_{(s,a,q)}^{\tau}$ of $\nabla_{(s,a,q)}^{\tau}$ are expressed locally as
\begin{equation}
\tilde{\Omega}_{(s,a,q)}^{\tau}=-2\pi \mathbf{i}dx^t \left( \frac{\partial s}{\partial x}(x) \right)^t dy-\pi \mathbf{i} dx^t \tau dx \label{loccurv}
\end{equation}
and
\begin{equation*}
\Omega_{(s,a,q)}^{\tau}=-2\pi \mathbf{i}dx^t \left( \frac{\partial s}{\partial x}(x) \right)^t dy,
\end{equation*}
respectively. Here, as explained in the above, note that the 2-form 
\begin{equation*}
-\pi \mathbf{i}dx^t \tau dx
\end{equation*}
in the local expression (\ref{loccurv}) is the 1-connection over $\mathcal{G}(I,\mathcal{O},\xi)$ which is compatible with the 0-connection $\nabla_{\tau}$ (a B-field). As mentioned in Remark \ref{i_c_r}, for an $\alpha$-twisted smooth complex vector bundle with a twisted connection (this $\alpha$ denotes the 2-cocycle which defines a given gerbe), we call the twisted connection a twisted integrable connection on the $\alpha$-twisted smooth complex vector bundle if the (0,2)-part of its curvature form vanishes (see \cite[Lemma 2.18, Lemma 2.19]{k-htwisted}). By direct calculations, the (0,2)-part of $\Omega_{(s,a,q)}^{\tau}$ turns out that
\begin{equation*}
\left( \Omega_{(s,a,q)}^{\tau} \right)^{(0,2)}=2\pi \mathbf{i} d\bar{z}^t ((T-\bar{T})^{-1})^t \left( \frac{\partial s}{\partial x}(x)T \right)^t (T-\bar{T})^{-1} d\bar{z}.
\end{equation*}
Thus, $\left( \Omega_{(s,a,q)}^{\tau} \right)^{(0,2)}=0$ is equivalent to that $((T-\bar{T})^{-1})^t \left( \frac{\partial s}{\partial x}(x)T \right)^t (T-\bar{T})^{-1}$ is a symmetric matrix, i.e.,
\begin{equation*}
\frac{\partial s}{\partial x}(x)T=\left( \frac{\partial s}{\partial x}(x)T \right)^t,
\end{equation*}
and then, by the local expression (\ref{loccurv}), we can easily verify that the (0,2)-part of the ``local'' curvature form $\tilde{\Omega}_{(s,a,q)}^{\tau}$ is expressed locally as
\begin{equation*}
\left( \tilde{\Omega}_{(s,a,q)}^{\tau} \right)^{(0,2)}=-\pi \mathbf{i} d\bar{z}^t ((T-\bar{T})^{-1})^t T^t \tau T(T-\bar{T})^{-1} d\bar{z}.
\end{equation*}
\end{proof}

Now, we define a dg-category $DG_{X_{\mathcal{G}_{\tau}^{\nabla}}^n}$ as follows. This definition is a natural generalization of the construction which is discussed in \cite{kajiura, exact} etc. to the twisted case. The objects are twisted holomorphic line bundles with twisted integrable connections $E_{(s,a,q)}^{\nabla}(\tau)$. For any two objects $E_{(s^1,a^1,q^1)}^{\nabla}(\tau)$ and $E_{(s^2,a^2,q^2)}^{\nabla}(\tau)$, the space of morphisms is defined by
\begin{align*}
&\mathrm{Hom}_{DG_{X_{\mathcal{G}_{\tau}^{\nabla}}^n}}\left( E_{(s^1,a^1,q^1)}^{\nabla}(\tau), E_{(s^2,a^2,q^2)}^{\nabla}(\tau) \right)=\bigoplus_{r\in \mathbb{N}\cup \{ 0 \}}\mathrm{Hom}_{DG_{X_{\mathcal{G}_{\tau}^{\nabla}}^n}}^r\left( E_{(s^1,a^1,q^1)}^{\nabla}(\tau), E_{(s^2,a^2,q^2)}^{\nabla}(\tau) \right), \\
&\mathrm{Hom}_{DG_{X_{\mathcal{G}_{\tau}^{\nabla}}^n}}^r\left( E_{(s^1,a^1,q^1)}^{\nabla}(\tau), E_{(s^2,a^2,q^2)}^{\nabla}(\tau) \right)=\Gamma \left( E_{(s^1,a^1,q^1)}^{\nabla}(\tau), E_{(s^2,a^2,q^2)}^{\nabla}(\tau) \right)\bigotimes_{C^{\infty}\left( X_{\mathcal{G}_{\tau}^{\nabla}}^n \right)} \Omega^{(0,r)}\left( X_{\mathcal{G}_{\tau}^{\nabla}}^n \right),
\end{align*}
where $\Gamma(E_{(s^1,a^1,q^1)}^{\nabla}(\tau), E_{(s^2,a^2,q^2)}^{\nabla}(\tau))$ is the space of homomorphisms from $E_{(s^1,a^1,q^1)}^{\nabla}(\tau)$ to $E_{(s^2,a^2,q^2)}^{\nabla}(\tau)$ (as twisted holomorphic line bundles with twisted integrable connections), and $\Omega^{(0,r)}(X_{\mathcal{G}_{\tau}^{\nabla}}^n)$ is the space of (0,$r$)-forms on $X_{\mathcal{G}_{\tau}^{\nabla}}^n$. This becomes a $\mathbb{Z}$-graded vector space over $\mathbb{C}$ by regarding $r\in \mathbb{N}\cup \{ 0 \}$ which is the degree of anti-holomorphic differential forms as the grading. We decompose $\nabla_{(s,a,q)}^{\tau}$ into its holomorphic part and anti-holomorphic part $\nabla_{(s,a,q)}^{\tau}=(\nabla_{(s,a,q)}^{\tau})^{(1,0)}+(\nabla_{(s,a,q)}^{\tau})^{(0,1)}$ on each $O_m^l$, and define a linear map
\begin{equation*}
d_{12}^r : \mathrm{Hom}_{DG_{X_{\mathcal{G}_{\tau}^{\nabla}}^n}}^r\left( E_{(s^1,a^1,q^1)}^{\nabla}(\tau), E_{(s^2,a^2,q^2)}^{\nabla}(\tau) \right) \to \mathrm{Hom}_{DG_{X_{\mathcal{G}_{\tau}^{\nabla}}^n}}^{r+1}\left( E_{(s^1,a^1,q^1)}^{\nabla}(\tau), E_{(s^2,a^2,q^2)}^{\nabla}(\tau) \right)
\end{equation*}
by
\begin{equation*}
d_{12}^r\left( \left\{ \phi_{(l;m)}^{12} \right\}_{(l;m)\in I} \right)=\left\{ \left( \nabla_{(s^2,a^2,q^2)}^{\tau} \right)^{(0,1)}\phi_{(l;m)}^{12}-(-1)^r \phi_{(l;m)}^{12}\left( \nabla_{(s^1,a^1,q^1)}^{\tau} \right)^{(0,1)} \right\}_{(l;m)\in I}
\end{equation*}
for each $r\in \mathbb{N}\cup \{ 0 \}$. Moreover, we define the composition
\begin{align*}
m : \ &\mathrm{Hom}_{DG_{X_{\mathcal{G}_{\tau}^{\nabla}}^n}}^{r_2}\left( E_{(s^2,a^2,q^2)}^{\nabla}(\tau), E_{(s^3,a^3,q^3)}^{\nabla}(\tau) \right)\bigotimes \mathrm{Hom}_{DG_{X_{\mathcal{G}_{\tau}^{\nabla}}^n}}^{r_1}\left( E_{(s^1,a^1,q^1)}^{\nabla}(\tau), E_{(s^2,a^2,q^2)}^{\nabla}(\tau) \right) \\
&\to \mathrm{Hom}_{DG_{X_{\mathcal{G}_{\tau}^{\nabla}}^n}}^{r_1+r_2}\left( E_{(s^1,a^1,q^1)}^{\nabla}(\tau), E_{(s^3,a^3,q^3)}^{\nabla}(\tau) \right)
\end{align*}
for each $(r_1,r_2)\in (\mathbb{N}\cup \{ 0 \})\times (\mathbb{N}\cup \{ 0 \})$ by
\begin{equation*}
m\left( \left\{ \phi_{(l;m)}^{23} \right\}_{(l;m)\in I}, \left\{ \phi_{(l;m)}^{12} \right\}_{(l;m)\in I} \right)=\left\{ \phi_{(l;m)}^{23} \wedge \phi_{(l;m)}^{12} \right\}_{(l;m)\in I}.
\end{equation*}
Then, the following theorem holds.
\begin{theo} \label{maintheorem1}
Twisted holomorphic line bundles with twisted integrable connections $E_{(s,a,q)}^{\nabla}(\tau)$ forms a dg-category $DG_{X_{\mathcal{G}_{\tau}^{\nabla}}^n}$.
\end{theo}
\begin{proof}
We first check that the linear map $d_{12}^r$ is a differential for an arbitrary fixed $r\in \mathbb{N}\cup \{ 0 \}$. For each $(l;m)\in I$, although we can locally calculate
\begin{align*}
d_{12}^r\left( \phi_{(l;m)}^{12} \right)&=\left( \nabla_{(s^2,a^2,q^2)}^{\tau} \right)^{(0,1)}\phi_{(l;m)}^{12}-(-1)^r \phi_{(l;m)}^{12}\left( \nabla_{(s^1,a^1,q^1)}^{\tau} \right)^{(0,1)} \\
&=\bar{\partial}\phi_{(l;m)}^{12}+\left( \omega_{(s^2,a^2,q^2)}^{\tau} \right)^{(0,1)}\wedge \phi_{(l;m)}^{12}-(-1)^r \phi_{(l;m)}^{12}\wedge \left( \omega_{(s^1,a^1,q^1)}^{\tau} \right)^{(0,1)},
\end{align*}
in fact, this is equal to
\begin{equation*}
\bar{\partial}\phi_{(l;m)}^{12}+\omega_{(s^2,a^2,q^2)}^{(0,1)}\wedge \phi_{(l;m)}^{12}-(-1)^r \phi_{(l;m)}^{12}\wedge \omega_{(s^1,a^1,q^1)}^{(0,1)}
\end{equation*}
since
\begin{align*}
&\Bigl( -\pi \mathbf{i}x^t \tau dx \Bigr)\wedge \phi_{(l;m)}^{12}-(-1)^r \phi_{(l;m)}^{12}\wedge \Bigl( -\pi \mathbf{i}x^t \tau dx \Bigr) \\
&=\Bigl( -\pi \mathbf{i}x^t \tau dx \Bigr)\wedge \phi_{(l;m)}^{12}-(-1)^r (-1)^r \Bigl( -\pi \mathbf{i}x^t \tau dx \Bigr)\wedge \phi_{(l;m)}^{12} \\
&=0,
\end{align*}
where $\{ \phi_{(l;m)}^{12} \}_{(l;m)\in I}\in \mathrm{Hom}_{DG_{X_{\mathcal{G}_{\tau}^{\nabla}}^n}}^r( E_{(s^1,a^1,q^1)}^{\nabla}(\tau), E_{(s^2,a^2,q^2)}^{\nabla}(\tau))$. It is easy to verify that the morphism $d_{12}^r(\{ \phi_{(l;m)}^{12} \}_{(l;m)\in I})$ is compatible with the families of transition functions of $E_{(s^1,a^1,q^1)}^{\nabla}(\tau)$ and $E_{(s^2,a^2,q^2)}^{\nabla}(\tau)$. Therefore, for each $(l;m)\in I$, $(d_{12}^r)^2 (\phi_{(l;m)}^{12})$ is expressed locally as
\begin{align}
&\left( d_{12}^r \right)^2 \left( \phi_{(l;m)}^{12} \right) \notag \\ 
&=d_{12}^r \left( \bar{\partial}\phi_{(l;m)}^{12}+\left( \omega_{(s^2,a^2,q^2)}^{\tau} \right)^{(0,1)}\wedge \phi_{(l;m)}^{12}-(-1)^r \phi_{(l;m)}^{12}\wedge \left( \omega_{(s^1,a^1,q^1)}^{\tau} \right)^{(0,1)} \right) \notag \\
&=\left( \Omega_{(s^2,a^2,q^2)}^{\tau} \right)^{(0,2)}\wedge \phi_{(l;m)}^{12}-\phi_{(l;m)}^{12}\wedge \left( \Omega_{(s^1,a^1,q^1)}^{\tau} \right)^{(0,2)}, \label{dg1}
\end{align}
and the integrability of $\nabla_{(s^1,a^1,q^1)}^{\tau}$ and $\nabla_{(s^2,a^2,q^2)}^{\tau}$ implies that the formula (\ref{dg1}) vanishes, i.e.,
\begin{equation*}
\left( d_{12}^r \right)^2 \left( \left\{ \phi_{(l;m)}^{12} \right\}_{(l;m)\in I} \right)=0.
\end{equation*}
Hence, the linear map $d_{12}^r$ is a differential for an arbitrary fixed $r\in \mathbb{N}\cup \{ 0 \}$.

Next, we check that the above differential and the composition $m$ of morphisms satisfy the Leibniz rule. For any
\begin{align*}
&\left\{ \phi_{(l;m)}^{12} \right\}_{(l;m)\in I}\in \mathrm{Hom}_{DG_{X_{\mathcal{G}_{\tau}^{\nabla}}^n}}^{r_1}\left( E_{(s^1,a^1,q^1)}^{\nabla}(\tau), E_{(s^2,a^2,q^2)}^{\nabla}(\tau) \right), \\
&\left\{ \phi_{(l;m)}^{23} \right\}_{(l;m)\in I}\in \mathrm{Hom}_{DG_{X_{\mathcal{G}_{\tau}^{\nabla}}^n}}^{r_2}\left( E_{(s^2,a^2,q^2)}^{\nabla}(\tau), E_{(s^3,a^3,q^3)}^{\nabla}(\tau) \right),
\end{align*}
it is clear that the morphism $m(\{ \phi_{(l;m)}^{23} \}_{(l;m)\in I}, \{ \phi_{(l;m)}^{12} \}_{(l;m)\in I})$ is compatible with the families of transition functions of $E_{(s^1,a^1,q^1)}^{\nabla}(\tau)$ and $E_{(s^3,a^3,q^3)}^{\nabla}(\tau)$ by the definition of $m$, and on each $O_m^l$, we see that the relation
\begin{align*}
&d_{23}^{r_2}\left( \phi_{(l;m)}^{23} \right)\wedge \phi_{(l;m)}^{12}+(-1)^{r_2}\phi_{(l;m)}^{23}\wedge d_{12}^{r_1}\left( \phi_{(l;m)}^{12} \right) \\
&=\left( \bar{\partial}\phi_{(l;m)}^{23}\wedge \phi_{(l;m)}^{12}+\left( \omega_{(s^3,a^3,q^3)}^{\tau} \right)^{(0,1)}\wedge \phi_{(l;m)}^{23}\wedge \phi_{(l;m)}^{12}-(-1)^{r_2}\phi_{(l;m)}^{23}\wedge \left( \omega_{(s^2,a^2,q^2)}^{\tau} \right)^{(0,1)}\wedge \phi_{(l;m)}^{12} \right) \\
&\hspace{4mm} +(-1)^{r_2}\left( \phi_{(l;m)}^{23}\wedge \bar{\partial}\phi_{(l;m)}^{12}+\phi_{(l;m)}^{23}\wedge \biggl( \omega_{(s^2,a^2,q^2)}^{\tau} \right)^{(0,1)}\wedge \phi_{(l;m)}^{12} \\
&\hspace{4mm} -(-1)^{r_1}\phi_{(l;m)}^{23}\wedge \phi_{(l;m)}^{12}\wedge \left( \omega_{(s^1,a^1,q^1)}^{\tau} \right)^{(0,1)} \biggr) \\
&=\bar{\partial}\left( \phi_{(l;m)}^{23}\wedge \phi_{(l;m)}^{12} \right)+\left( \omega_{(s^3,a^3,q^3)}^{\tau} \right)^{(0,1)}\wedge \left( \phi_{(l;m)}^{23}\wedge \phi_{(l;m)}^{12} \right) \\
&\hspace{4mm} -(-1)^{r_1+r_2}\left( \phi_{(l;m)}^{23}\wedge \phi_{(l;m)}^{12} \right)\wedge \left( \omega_{(s^1,a^1,q^1)}^{\tau} \right)^{(0,1)} \\
&=d_{13}^{r_1+r_2}\left( \phi_{(l;m)}^{23}\wedge \phi_{(l;m)}^{12} \right)
\end{align*}
holds. This relation states that the differential and the product structure $m$ satisfy the Leibniz rule.

Thus, indeed, this $DG_{X_{\mathcal{G}_{\tau}^{\nabla}}^n}$ forms a dg-category.
\end{proof}
As a corollary of Theorem \ref{maintheorem1}, we see that holomorphic line bundles with integrable connections $E_{(s,a,q)}^{\nabla}$ on $X^n$ forms a dg-category $DG_{X^n}:=DG_{X_{\mathcal{G}_O^{\nabla}}^n}$ when we consider the case that $\mathcal{G}_{\tau}^{\nabla}$ is trivial, i.e., $\tau=O$, and actually, such a dg-category is used in various studies (see \cite{kajiura, exact} etc.). We can regard the dg-category $DG_{X_{\mathcal{G}_{\tau}^{\nabla}}^n}$ as the deformation of the dg-category $DG_{X^n}$ associated to the deformation from $X^n$ to $X_{\mathcal{G}_{\tau}^{\nabla}}^n$ by $\mathcal{G}_{\tau}^{\nabla}$, this can also be understood as the equivalence
\begin{equation}
E_{(0,0,0)}^{\nabla}(\tau)\otimes - : DG_{X^n}\stackrel{\sim}{\to} DG_{X_{\mathcal{G}_{\tau}^{\nabla}}^n}. \label{dgequiv}
\end{equation}
In particular, for a given algebraic variety $X$ over $\mathbb{C}$, roughly speaking, if $\alpha$ is an element in the Brauer group of $X$, there exists an $\alpha$-twisted vector bundle $\mathcal{E}_{\alpha}\to X$ and we can consider the sheaf of Azumaya algebras $\mathcal{A}:=\mathrm{End}(\mathcal{E}_{\alpha})$. Then, it is known that there exists an equivalence between the abelian category $Coh(X,\alpha)$ of $\alpha$-twisted sheaves on $X$ and the abelian category $Coh(X,\mathcal{A})$ of right coherent $\mathcal{A}$-modules on $X$:
\begin{equation}
\mathrm{Hom}(\mathcal{E}_{\alpha},-) : Coh(X,\alpha)\stackrel{\sim}{\to} Coh(X,\mathcal{A}). \label{azumaya}
\end{equation}
These facts are explained in \cite{cal} for example (cf. \cite{kap1, kap2}). In this paper, although we do not assume that $X^n$ is an abelian variety (we regard $X^n$ as a general complex torus), it seems that the deformation (\ref{dgequiv}) can also be interpreted as an analogue of the equivalence (\ref{azumaya}).

\subsection{A problem}
The goal of subsection 5.3 and subsection 5.4 is to construct the deformation of objects $\mathscr{L}_{(s,a,q)}^{\nabla}\in Fuk_{\rm sub}(\check{X}^n)$ associated to the deformation from $\check{X}^n$ to $\check{X}_{\mathcal{G}_{\tau}^{\nabla}}^n$. In particular, we need to consider the deformation of unitary local systems $\mathcal{L}_{(s,a,q)}^{\nabla}$ in order to achieve this goal. However, when we discuss the deformation of objects $\mathscr{L}_{(s,a,q)}^{\nabla}\in Fuk_{\rm sub}(\check{X}^n)$, in general, there exists an alternating matrix $\tau \in M(n;\mathbb{R})$ such that any deformations of $\mathcal{L}_{(s,a,q)}^{\nabla}$ cannot be constructed in the framework of the usual Fukaya categories. The purposes of this subsection are to explain this problem, and to give the outline of an idea to solve it.

Let us again recall the usual definition of objects of the Fukaya category over a given symplectic manifold $M$ with a B-field $2\pi B$. An object of the Fukaya category over this complexified symplectic manifold consists of a Lagrangian submanifold $L$ in $M$ and a smooth complex line bundle $\mathcal{L}\to L$ with a unitary connection $\nabla_{\mathcal{L}}$ such that
\begin{equation}
\Omega_{\nabla_{\mathcal{L}}}=\Bigl. -2\pi \mathbf{i}B \Bigr|_L, \label{usualdef}
\end{equation}
where $\Omega_{\nabla_{\mathcal{L}}}$ denotes the curvature form of $\nabla_{\mathcal{L}}$. In particular, there exists such a unitary connection $\nabla_{\mathcal{L}}$ if and only if $[B]\in H^2(L;\mathbb{Z})$, namely, there does not exist a unitary connection $\nabla_{\mathcal{L}}$ which satisfies the condition (\ref{usualdef}) on each Lagrangian submanifold $L$ such that $[B]\not \in H^2(L;\mathbb{Z})$.

For later convenience, we prepare some notations. Note that $\tau \in M(n;\mathbb{R})$ is an alternating matrix. For each component $\tau_{ij}$ of $\tau \in M(n;\mathbb{R})$ such that $1\leq i<j\leq n$, we set
\begin{equation*}
\lfloor \tau_{ij} \rfloor :=\mathrm{max}\Bigl\{ m\in \mathbb{Z} \ | \ m\leq \tau_{ij} \Bigr\},
\end{equation*}
and let us define two alternating matrices $\tau_{\mathbb{Z}}\in M(n;\mathbb{Z})$ and $\tau_{\rm frac}\in M(n;\mathbb{R})$ by
\begin{align*}
&(\tau_{\mathbb{Z}})_{ij}=\lfloor \tau_{ij} \rfloor, \ (\tau_{\mathbb{Z}})_{ji}=-\lfloor \tau_{ij} \rfloor, \\
&(\tau_{\rm frac})_{ij}=\tau_{ij}-\lfloor \tau_{ij} \rfloor, \ (\tau_{\rm frac})_{ji}=-\tau_{ij}+\lfloor \tau_{ij} \rfloor,
\end{align*}
respectively. Then, we have
\begin{equation*}
\tau=\tau_{\mathbb{Z}}+\tau_{\rm frac}.
\end{equation*}
By using these notations, the B-field $B_{\tau}^{\vee}$ on $\check{X}_{\mathcal{G}_{\tau}^{\nabla}}^n$ is expressed locally as
\begin{equation*}
B_{\tau}^{\vee}=B^{\vee}+\pi d\check{x}^t \tau d\check{x}=B^{\vee}+\pi d\check{x}^t \tau_{\mathbb{Z}} d\check{x} +\pi d\check{x}^t \tau_{\rm frac} d\check{x}.
\end{equation*}

Here, we consider the case $\tau \in M(n;\mathbb{Z})$, i.e., $\tau_{\rm frac}=O$. Then, for the B-field $B^{\vee}+\pi d\check{x}^t \tau_{\mathbb{Z}} d\check{x}$, since 
\begin{equation*}
\Bigl.B^{\vee} \Bigr|_{L_{(s,a)}}=2\pi d\check{x}^t B_{\rm mat}^{\vee} \frac{\partial s}{\partial \check{x}}(\check{x}) d\check{x},
\end{equation*}
if we take the Lagrangian submanifold $L_{(s,a)}$ in $(\mathbb{R}^{2n}/\mathbb{Z}^{2n}, \omega^{\vee})$ which is determined by each $s$ satisfying the relation
\begin{equation*}
B_{\rm mat}^{\vee} \frac{\partial s}{\partial \check{x}}(\check{x})=\left( B_{\rm mat}^{\vee} \frac{\partial s}{\partial \check{x}}(\check{x}) \right)^t,
\end{equation*}
then the requirement (\ref{usualdef}) is satisfied:\footnote{It is natural to consider that $L_{(s,a)}$ is not deformed since $\omega^{\vee}$ is preserved in the deformation from $\check{X}^n$ to $\check{X}_{\mathcal{G}_{\tau}^{\nabla}}^n$.}
\begin{equation*}
\left[ d\check{x}^t B_{\rm mat}^{\vee} d\check{y}+\frac{1}{2}d\check{x}^t \tau_{\mathbb{Z}} d\check{x} \right]=\left[ \frac{1}{2}d\check{x}^t \tau_{\mathbb{Z}} d\check{x} \right]\in H^2(L_{(s,a)};\mathbb{Z}).
\end{equation*}
Hence, when we focus on such Lagrangian submanifolds $L_{(s,a)}$, we can expect that the deformation of objects $\mathscr{L}_{(s,a,q)}^{\nabla}\in Fuk_{\rm sub}(\check{X}^n)$ can be constructed in the framework of the usual Fukaya categories.

However, in the case $\tau \not \in M(n;\mathbb{Z})$, i.e., $\tau_{\rm frac}\not=O$, for the B-field $B^{\vee}+\pi d\check{x}^t \tau d\check{x}$, it is easy to check that
\begin{equation*}
\left[ d\check{x}^t B_{\rm mat}^{\vee} d\check{y}+\frac{1}{2}d\check{x}^t \tau d\check{x} \right]\not \in H^2(L_{(s,a)};\mathbb{Z})
\end{equation*}
holds for any Lagrangian submanifold $L_{(s,a)}$ in $(\mathbb{R}^{2n}/\mathbb{Z}^{2n}, \omega^{\vee})$ even though we can consider the deformation $E_{(s,a,q)}^{\nabla}(\tau)$ of each $E_{(s,a,q)}^{\nabla}$ depending on $\tau \not \in M(n;\mathbb{Z})$. Thus, we can not discuss the deformation of objects $\mathscr{L}_{(s,a,q)}^{\nabla}\in Fuk_{\rm sub}(\check{X}^n)$ in the usual framework over $\check{X}_{\mathcal{G}_{\tau}^{\nabla}}^n$ under the assumption $\tau \not \in M(n;\mathbb{Z})$.

Our proposal to solve this problem is to consider the twisting of the complexified symplectic torus $(\mathbb{R}^{2n}/\mathbb{Z}^{2n}, \omega^{\vee}-\mathbf{i}(B^{\vee}+\pi d\check{x}^t \tau_{\mathbb{Z}} d\check{x}))$ by the flat gerbe $\check{\mathcal{G}}_{\tau_{\rm frac}}^{\nabla}$ associated to $\tau_{\rm frac}\in M(n;\mathbb{R})$. In other words, our idea is to employ ``twisted'' line bundles associated to the flat gerbe $\check{\mathcal{G}}_{\tau_{\rm frac}}^{\nabla}$ instead of ``usual'' line bundles on Lagrangian submanifolds. Strictly speaking, we do not discuss the compatibility of our proposal and some structures required to define (an analogue of) the Fukaya category over $\check{X}_{\check{\mathcal{G}}_{\tau}^{\nabla}}^n$, so in this sense, our proposal is just an idea.

\subsection{A deformation of objects of $Fuk_{\rm sub}(\check{X}^n)$ and the homological mirror symmetry}
The purpose of this subsection is to explain the details of the proposal which is mentioned in subsection 5.3, namely, to discuss the deformation of objects $\mathscr{L}_{(s,a,q)}^{\nabla}\in Fuk_{\rm sub}(\check{X}^n)$ which is compatible with the counterpart in the mirror dual complex geometry side.

Let us first define a flat gerbe $\check{\mathcal{G}}_{\tau_{\rm frac}}^{\nabla}$ over the complexified symplectic torus 
\begin{equation*}
\check{X}^n(\tau_{\mathbb{Z}}):=\Bigl( \mathbb{R}^{2n}/\mathbb{Z}^{2n}, \ \omega^{\vee}-\mathbf{i}(B^{\vee}+\pi d\check{x}^t \tau_{\mathbb{Z}} d\check{x}) \Bigr)
\end{equation*}
depending on $\tau_{\rm frac}\in M(n;\mathbb{R})$ as follows. We fix an $\epsilon >0$ small enough and take an open covering $\{ \check{O}_m^l:=\check{O}_{m_1\cdots m_n}^{l_1\cdots l_n} \}_{(l;m)\in I}$ of $\check{X}^n$ which is obtained by replacing $\left( \begin{array}{cc} x \\ y \end{array} \right)$ with $\left( \begin{array}{cc} \check{x} \\ \check{y} \end{array} \right)$, i.e.,
\begin{align*}
\check{O}_{m_1\cdots m_n}^{l_1\cdots l_n}:=\biggl\{ \left( \begin{array}{ccc}\check{x} \\ \check{y} \end{array} \right)\in \check{X}^n \ \Bigl. \Bigr| \ &\frac{l_j-1}{3}-\epsilon <x^j <\frac{l_j}{3}+\epsilon, \\
&\frac{m_k-1}{3}-\epsilon <y^k <\frac{m_k}{3}+\epsilon, \ j,k=1,\cdots, n \biggr\},
\end{align*}
and in particular, we can also regard this $\{ \check{O}_m^l \}_{(l;m)\in I}$ as an open covering of $\check{X}^n(\tau_{\mathbb{Z}})$. Let us consider the trivial gerbe $\check{\mathcal{G}}(I,\mathcal{O},\xi)$ on $\check{X}^n(\tau_{\mathbb{Z}})$ (similarly as in the case of $\mathcal{G}(I,\mathcal{O},\xi)$, we use the notations $\mathcal{O}$ and $\xi$ here):
\begin{align*}
&\mathcal{O}:=\Bigl\{ \mathcal{O}_{mm'}^{ll'}\to \check{O}_{mm'}^{ll'} \Bigr\}_{(l;m),(l';m')\in I}, \\ &\xi:=\Bigl\{ \xi_{mm'm''}^{ll'l''}=1\in \Gamma(\check{O}_{mm'm''}^{ll'l''};\mathcal{O}_{mm'm''}^{ll'l''}) \Bigr\}_{(l;m),(l';m'),(l'';m'')\in I}.
\end{align*}
Similarly, a 0-connection $\check{\nabla}_{\tau_{\rm frac}}:=\{ ^{\tau_{\rm frac}}\!\check{\nabla}_{mm'}^{ll'}:=d-\mathbf{i}^{\tau_{\rm frac}}\!\check{\omega}_{mm'}^{ll'} \}_{(l;m),(l';m')\in I}$ over $\check{\mathcal{G}}(I,\mathcal{O},\xi)$ is also defined by replacing $\tau$ and $x$ with $\tau_{\rm frac}$ and $\check{x}$, respectively, namely, a 1-form $^{\tau_{\rm frac}}\!\check{\omega}_{mm'}^{ll'}$ on the open set $\check{O}_{mm'}^{ll'}$ is defined by
\begin{equation*}
^{\tau_{\rm frac}}\!\check{\omega}_{mm'}^{ll'}=
\begin{cases}
\pi (\tau_{\rm frac})_j^t d\check{x} & (l_j=1, l_j'=3, l_h=l_h' (h\not=j), m=m') \\ 
\pi (\tau_{\rm frac})_j^t d\check{x} & (l_j=1, l_j'=3, l_h=l_h' (h\not=j), m_k=1, m_k'=3, m_i=m_i' (i\not=k)) \\ 
\pi (\tau_{\rm frac})_j^t d\check{x} & (l_j=1, l_j'=3, l_h=l_h' (h\not=j), m_k=3, m_k'=1, m_i=m_i' (i\not=k)) \\ 
0 & (\mathrm{otherwise}),
\end{cases}
\end{equation*}
where 
\begin{equation*}
(\tau_{\rm frac})_j:=\left( \begin{array}{ccc} (\tau_{\rm frac})_{j1} \\ \vdots \\ (\tau_{\rm frac})_{jn} \end{array} \right)\in \mathbb{R}^n,
\end{equation*}
$j$, $k=1$, $\cdots$, $n$, and we assume the condition $^{\tau_{\rm frac}}\!\check{\omega}_{m'm}^{l'l}=-^{\tau_{\rm frac}}\!\check{\omega}_{mm'}^{ll'}$. Moreover, although we drop the index $(l;m)\in I$ for simplicity, the family
\begin{equation*}
\Bigl\{ -\pi \mathbf{i}d\check{x}^t \tau_{\rm frac} d\check{x} \Bigr\}_{(l;m)\in I}
\end{equation*}
gives a 1-connection over $\check{\mathcal{G}}(I,\mathcal{O},\xi)$ which is compatible with the 0-connection $\check{\nabla}_{\tau_{\rm frac}}$. We denote the trivial gerbe $\check{\mathcal{G}}(I,\mathcal{O},\xi)$ with the flat connection $(\check{\nabla}_{\tau_{\rm frac}},\{ -\pi \mathbf{i}d\check{x}^t \tau_{\rm frac} d\check{x} \}_{(l;m)\in I})$ by $\check{\mathcal{G}}_{\tau_{\rm frac}}^{\nabla}$:
\begin{equation*}
\check{\mathcal{G}}_{\tau_{\rm frac}}^{\nabla}:=\Bigl( \check{\mathcal{G}}(I,\mathcal{O},\xi), \ \check{\nabla}_{\tau_{\rm frac}}, \ \Bigl\{ -\pi \mathbf{i}d\check{x}^t \tau_{\rm frac} d\check{x} \Bigr\}_{(l;m)\in I} \Bigr).
\end{equation*}
By using this flat gerbe $\check{\mathcal{G}}_{\tau_{\rm frac}}^{\nabla}$, similarly as in the deformation $X_{\mathcal{G}_{\tau}^{\nabla}}^n$ of $X^n$ by $\mathcal{G}_{\tau}^{\nabla}$, we can consider the deformation $\check{X}^n(\tau_{\mathbb{Z}})_{\check{\mathcal{G}}_{\tau_{\rm frac}}^{\nabla}}$ of $\check{X}^n(\tau_{\mathbb{Z}})$ by $\check{\mathcal{G}}_{\tau_{\rm frac}}^{\nabla}$, and $\check{X}_{\mathcal{G}_{\tau}^{\nabla}}^n$ defined in section 4 can be identified with this $\check{X}^n(\tau_{\mathbb{Z}})_{\check{\mathcal{G}}_{\tau_{\rm frac}}^{\nabla}}$. By concerning this fact, hereafter, we treat $\check{X}^n(\tau_{\mathbb{Z}})_{\check{\mathcal{G}}_{\tau_{\rm frac}}^{\nabla}}$ as a mirror partner of $X_{\mathcal{G}_{\tau}^{\nabla}}^n$: 
\begin{equation*}
\check{X}_{\mathcal{G}_{\tau}^{\nabla}}^n=\check{X}^n(\tau_{\mathbb{Z}})_{\check{\mathcal{G}}_{\tau_{\rm frac}}^{\nabla}}. 
\end{equation*}

Let us discuss the deformation of a given arbitrary object $\mathscr{L}_{(s,a,q)}^{\nabla}\in Fuk_{\rm sub}(\check{X}^n)$ associated to the deformation from $\check{X}^n$ to $\check{X}_{\mathcal{G}_{\tau}^{\nabla}}^n$. It is natural to consider that $L_{(s,a)}$ is not deformed and $\mathcal{L}_{(s,a,q)}^{\nabla}$ is only deformed depending on $\tau$ since $\omega^{\vee}$ is preserved and $B^{\vee}$ is twisted by using $\tau$ in the deformation from $\check{X}^n$ to $\check{X}_{\mathcal{G}_{\tau}^{\nabla}}^n$. Therefore, as candidates of Lagrangian submanifolds, we focus on $n$-dimensional submanifolds $L_{(s,a)}$ in $\check{X}_{\mathcal{G}_{\tau}^{\nabla}}^n$ similarly as in the case of $Fuk_{\rm sub}(\check{X}^n)$. Below, we construct the deformation of $\mathcal{L}_{(s,a,q)}^{\nabla}$. Let us consider an $id$-twisted smooth complex line bundle $\mathcal{L}_{(s,a,q)}^{\tau}\to L_{(s,a)}$ which is defined by a transition function expressed locally as 
\begin{equation*}
\psi \left( \begin{array}{cccccccc} x^1 \\ \vdots \\ x^i+1 \\ \vdots \\ x^n \\ s^1(\check{x})+a_{1i} \\ \vdots \\ s^n(\check{x})+a_{ni} \end{array} \right)=\mathrm{exp}\left( \pi \mathbf{i} (\tau_{\mathbb{Z}})_i^t \check{x} \right) \cdot \psi \left( \begin{array}{cccccccc} x^1 \\ \vdots \\ x^i \\ \vdots \\ x^n \\ s^1(\check{x}) \\ \vdots \\ s^n(\check{x}) \end{array} \right)
\end{equation*}
for each $i=1$, $\cdots$, $n$, where $\psi$ is a locally defined smooth section of $\mathcal{L}_{(s,a,q)}^{\tau}$, and
\begin{equation*}
(\tau_{\mathbb{Z}})_i:=\left( \begin{array}{ccc} (\tau_{\mathbb{Z}})_{i1} \\ \vdots \\ (\tau_{\mathbb{Z}})_{in} \end{array} \right)\in \mathbb{Z}^n.
\end{equation*}
We further consider a twisted connection on $\mathcal{L}_{(s,a,q)}^{\tau}$ which is expressed locally as
\begin{equation*}
\nabla_{\mathcal{L}_{(s,a,q)}^{\tau}}:=d-\pi \mathbf{i} \check{x}^t \tau_{\mathbb{Z}} d\check{x}-2\pi \mathbf{i} q^t d\check{x}-\pi \mathbf{i}\check{x}^t \tau_{\rm frac} d\check{x},
\end{equation*}
where 
\begin{equation*}
q:=\left( \begin{array}{ccc} q_1 \\ \vdots \\ q_n \end{array} \right)\in \mathbb{R}^n
\end{equation*}
is a constant vector, and $-\pi \mathbf{i}\check{x}^t \tau_{\rm frac} d\check{x}$ comes from the deformation of $\check{X}^n(\tau_{\mathbb{Z}})$ by $\mathcal{G}_{\tau_{\rm frac}}^{\nabla}$. It is easy to check that the local expression of the curvature form $\Omega_{\nabla_{\mathcal{L}_{(s,a,q)}^{\tau}}}:=\nabla_{\mathcal{L}_{(s,a,q)}^{\tau}}\circ \nabla_{\mathcal{L}_{(s,a,q)}^{\tau}}$ of $\nabla_{\mathcal{L}_{(s,a,q)}^{\tau}}$ is given by
\begin{equation*}
\Omega_{\nabla_{\mathcal{L}_{(s,a,q)}^{\tau}}}=-\pi \mathbf{i} d\check{x}^t \tau_{\mathbb{Z}} d\check{x}.
\end{equation*}
At least, this implies that $\tau_{\mathbb{Z}} \in M(n;\mathbb{Z})$ determines the first Chern character $ch_1(\mathcal{L}_{(s,a,q)}^{\tau})$ of $\mathcal{L}_{(s,a,q)}^{\tau}$ in the case $\tau\in M(n;\mathbb{Z})$, i.e., $\tau_{\rm frac}=O$:
\begin{equation*}
ch_1(\mathcal{L}_{(s,a,q)}^{\tau})=\frac{1}{2}d\check{x}^t \tau_{\mathbb{Z}} d\check{x}.
\end{equation*}
Hereafter, for an $id$-twisted smooth complex line bundle $\mathcal{L}_{(s,a,q)}^{\tau}$ and a twisted connection $\nabla_{\mathcal{L}_{(s,a,q)}^{\tau}}$ on $\mathcal{L}_{(s,a,q)}^{\tau}$, we set
\begin{equation*}
\mathcal{L}_{(s,a,q)}^{\nabla}(\tau):=\Bigl( \mathcal{L}_{(s,a,q)}^{\tau}, \ \nabla_{\mathcal{L}_{(s,a,q)}^{\tau}} \Bigr),
\end{equation*}
and let us denote a pair of an $n$-dimensional submanifold $L_{(s,a)}$ in $\check{X}_{\mathcal{G}_{\tau}^{\nabla}}^n$ and an $id$-twisted smooth complex line bundle with a twisted connection $\mathcal{L}_{(s,a,q)}^{\nabla}(\tau)$ by $\mathscr{L}_{(s,a,q)}^{\nabla}(\tau)$: 
\begin{equation*}
\mathscr{L}_{(s,a,q)}^{\nabla}(\tau):=\Bigl( L_{(s,a)}, \ \mathcal{L}_{(s,a,q)}^{\nabla}(\tau) \Bigr).
\end{equation*}
Then, we obtain the following.
\begin{proposition} \label{generalization}
For a given locally defined smooth function $s$ satisfying the relations \textup{(\ref{lagsection})} and a given constant vector $q\in \mathbb{R}^n$, the pair $\mathscr{L}_{(s,a,q)}^{\nabla}(\tau)$ satisfies the condition
\begin{equation}
L_{(s,a)} \ \mathrm{becomes} \ \mathrm{a} \ \mathrm{Lagrangian} \ \mathrm{submanifold} \ \mathrm{in} \ \check{X}_{\mathcal{G}_{\tau}^{\nabla}}^n \label{general1} 
\end{equation}
and the condition
\begin{equation}
\mathcal{L}_{(s,a,q)}^{\tau}\to L_{(s,a)} \ \mathrm{with} \ \nabla_{\mathcal{L}_{(s,a,q)}^{\tau}} \ \mathrm{satisfies} \ \tilde{\Omega}_{\nabla_{\mathcal{L}_{(s,a,q)}^{\tau}}}=\Bigl. -\mathbf{i}B_{\tau}^{\vee}\Bigr |_{L_{(s,a)}} \label{general2} 
\end{equation}
if and only if 
\begin{equation}
\frac{\partial s}{\partial \check{x}}(\check{x})T=\left( \frac{\partial s}{\partial \check{x}}(\check{x})T \right)^t. \label{symm_symp}
\end{equation}
holds, where $\tilde{\Omega}_{\nabla_{\mathcal{L}_{(s,a,q)}^{\tau}}}$ denotes the ``local'' curvature form of $\nabla_{\mathcal{L}_{(s,a,q)}^{\tau}}$.
\end{proposition}
\begin{proof}
We first consider the condition (\ref{general1}). By direct calculations, we see
\begin{align*}
\Bigl. \omega^{\vee} \Bigr |_{L_{(s,a)}}&=\Bigl. 2\pi d\check{x}^t \omega_{\rm mat}^{\vee} d\check{y} \Bigr |_{L_{(s,a)}} \\
&=2\pi d\check{x}^t \omega_{\rm mat}^{\vee} \frac{\partial s}{\partial \check{x}}(\check{x}) d\check{x},
\end{align*}
and this implies that $L_{(s,a)}$ becomes a Lagrangian submanifold in $\check{X}_{\mathcal{G}_{\tau}^{\nabla}}^n$, i.e., $\omega^{\vee}|_{L_{(s,a)}}=0$ holds if and only if
\begin{equation}
\omega_{\rm mat}^{\vee}\frac{\partial s}{\partial \check{x}}(\check{x})=\left( \omega_{\rm mat}^{\vee}\frac{\partial s}{\partial \check{x}}(\check{x}) \right)^t \label{twistedlag}
\end{equation}
holds. Next, we consider the condition (\ref{general2}). Note that $\tilde{\Omega}_{\nabla_{\mathcal{L}_{(s,a,q)}^{\tau}}}$ is expressed locally as
\begin{equation*}
\tilde{\Omega}_{\nabla_{\mathcal{L}_{(s,a,q)}^{\tau}}}=-\pi \mathbf{i}d\check{x}^t \tau_{\mathbb{Z}} d\check{x}-\pi \mathbf{i}d\check{x}^t \tau_{\rm frac} d\check{x}=-\pi \mathbf{i} d\check{x}^t \tau d\check{x}.
\end{equation*}
Hence, the restriction of $-\mathbf{i}B_{\tau}^{\vee}$ to $L_{(s,a)}$ turns out to be
\begin{align*}
\Bigl. -\mathbf{i}B_{\tau}^{\vee} \Bigr |_{L_{(s,a)}}&=\Bigl. -2\pi \mathbf{i}d\check{x}^t B_{\rm mat}^{\vee} d\check{y}-\pi \mathbf{i}d\check{x}^t \tau d\check{x} \Bigr |_{L_{(s,a)}} \\
&=-2\pi \mathbf{i}d\check{x}^t B_{\rm mat}^{\vee} \frac{\partial s}{\partial \check{x}}(\check{x}) d\check{x}-\pi \mathbf{i} d\check{x}^t \tau d\check{x} \\
&=-2\pi \mathbf{i}d\check{x}^t B_{\rm mat}^{\vee} \frac{\partial s}{\partial \check{x}}(\check{x}) d\check{x}+\tilde{\Omega}_{\nabla_{\mathcal{L}_{(s,a,q)}^{\tau}}},
\end{align*}
so the relation $\tilde{\Omega}_{\nabla_{\mathcal{L}_{(s,a,q)}^{\tau}}}=-\mathbf{i}B_{\tau}^{\vee}|_{L_{(s,a)}}$ holds if and only if 
\begin{equation}
B_{\rm mat}^{\vee}\frac{\partial s}{\partial \check{x}}(\check{x})=\left( B_{\rm mat}^{\vee}\frac{\partial s}{\partial \check{x}}(\check{x}) \right)^t \label{twistedb'}
\end{equation}
holds. The statement follows from the relation (\ref{twistedlag}) and the relation (\ref{twistedb'}).
\end{proof}
\begin{rem}
In the case $\tau \in M(n;\mathbb{Z})$, i.e., $\tau_{\rm frac}=O$, the ``local'' curvature form $\tilde{\Omega}_{\nabla_{\mathcal{L}_{(s,a,q)}^{\tau}}}$ coincides with the curvature form $\Omega_{\nabla_{\mathcal{L}_{(s,a,q)}^{\tau}}}$ since $\check{\mathcal{G}}_{\tau_{\rm frac}}^{\nabla}$ is trivial. Thus, in this case, the conditions \textup{(\ref{general1})}, \textup{(\ref{general2})} reproduce the definition of objects of the usual Fukaya categories.
\end{rem}
By comparing Proposition \ref{twhol} with Proposition \ref{generalization}, we immediately obtain the following.
\begin{proposition} \label{generalizationmirror}
For a given locally defined smooth function $s$ satisfying the relations \textup{(\ref{section})}, \textup{(\ref{lagsection})} and a given constant vector $q\in \mathbb{R}^n$, the twisted connection $\nabla_{(s,a,q)}^{\tau}$ is a twisted integrable connection on $E_{(s,a,q)}^{\tau}\to X_{\mathcal{G}_{\tau}^{\nabla}}^n$ if and only if the pair $\mathscr{L}_{(s,a,q)}^{\nabla}(\tau)$ satisfies the conditions \textup{(\ref{general1})}, \textup{(\ref{general2})}.
\end{proposition}
Clearly, Proposition \ref{generalizationmirror} indicates that
\begin{equation*}
E_{(s,a,q)}^{\nabla}(\tau)\in DG_{X_{\mathcal{G}_{\tau}^{\nabla}}^n}
\end{equation*}
and 
\begin{equation*}
\mathscr{L}_{(s,a,q)}^{\nabla}(\tau)
\end{equation*}
are mirror dual to each other under the conditions (\ref{symm_comp}), (\ref{symm_symp}). Moreover, when $E_{(s,a,q)}^{\nabla}(\tau)$ and $\mathscr{L}_{(s,a,q)}^{\nabla}(\tau)$ are mirror dual to each other, we see that both the (0,2)-part of the ``local'' curvature form $\tilde{\Omega}_{(s,a,q)}^{\tau}$ of $\nabla_{(s,a,q)}^{\tau}$ and the ``local'' curvature form $\tilde{\Omega}_{\nabla_{\mathcal{L}_{(s,a,q)}^{\tau}}}$ of $\nabla_{\mathcal{L}_{(s,a,q)}^{\nabla}}$ are completely determined by $\tau \in M(n;\mathbb{R})$. As mentioned in subsection 5.3, strictly speaking, in this paper, we do not discuss the compatibility of our proposal and some structures required to define (an analogue of) the Fukaya category over $\check{X}_{\mathcal{G}_{\tau}^{\nabla}}^n$, so in this sense, our proposal discussed in subsection 5.3 and subsection 5.4 is just an idea. However, at least, Proposition \ref{generalizationmirror} implies that to consider an analogue of (the full subcategory of) the Fukaya category over $\check{X}_{\mathcal{G}_{\tau}^{\nabla}}^n$ consisting of objects $\mathscr{L}_{(s,a,q)}^{\nabla}(\tau)$ satisfying the conditions (\ref{general1}), (\ref{general2}) is natural from the viewpoint of the deformation theory and the homological mirror symmetry conjecture for $(X_{\mathcal{G}_{\tau}^{\nabla}}^n,\check{X}_{\mathcal{G}_{\tau}^{\nabla}}^n)$.

Finally, we give a remark. Although we focus on the division $\tau =\tau_{\mathbb{Z}}+\tau_{\rm frac}$ in the above, in general, the choice of such a division of $\tau$ is not unique. For instance, we can consider the trivial division $\tau =O+\tau$ instead of the division $\tau =\tau_{\mathbb{Z}}+\tau_{\rm frac}$\footnote{Of course, although the division $\tau =\tau_{\mathbb{Z}}+\tau_{\rm frac}$ coincides with the trivial division of $\tau$ if $\tau_{\mathbb{Z}}=O$, they are mutually distinct in general.}. Moreover, according to this trivial division of $\tau$, we can twist $\check{X}^n$ by using the flat gerbe $\check{\mathcal{G}}_{\tau}^{\nabla}$ which is obtained by replacing $\tau_{\rm frac}$ in the definition of $\check{\mathcal{G}}_{\tau_{\rm frac}}^{\nabla}$ with $\tau$. In particular, $\check{X}_{\mathcal{G}_{\tau}^{\nabla}}^n$ can also be identified with the deformation of $\check{X}^n$ by this $\check{\mathcal{G}}_{\tau}^{\nabla}$:
\begin{equation*}
\check{X}_{\mathcal{G}_{\tau}^{\nabla}}^n=\check{X}_{\check{\mathcal{G}}_{\tau}^{\nabla}}^n.
\end{equation*}
However, it can be expected that our construction explained in the above does not depend on the choice of such a division of $\tau$. Here, we take the trivial division $\tau =O+\tau$ as an example of a division of $\tau$ which differs from the division $\tau =\tau_{\mathbb{Z}}+\tau_{\rm frac}$, and below, describe this perspective by comparing them.

When we focus on the twisting of $\check{X}^n$ by $\check{\mathcal{G}}_{\tau}^{\nabla}$, each object $\mathcal{L}_{(s,a,q)}^{\nabla}$ is deformed to the $id$-twisted trivial complex line bundle $\mathcal{F}_{(s,a,q)}^{\tau}$ with the twisted flat connection
\begin{equation*}
\nabla_{\mathcal{F}_{(s,a,q)}^{\tau}}:=d-2\pi \mathbf{i}q^t d\check{x}-\pi \mathbf{i} \check{x}^t \tau d\check{x}
\end{equation*}
(it is natural to consider that $L_{(s,a)}$ is preserved under this deformation). We set
\begin{equation*}
\mathcal{F}_{(s,a,q)}^{\nabla}(\tau):=\Bigl( \mathcal{F}_{(s,a,q)}^{\tau}, \ \nabla_{\mathcal{F}_{(s,a,q)}^{\tau}} \Bigr),
\end{equation*}
and denote the pair of $L_{(s,a)}$ and $\mathcal{F}_{(s,a,q)}^{\nabla}(\tau)$ by $\mathscr{F}_{(s,a,q)}^{\nabla}(\tau)$:
\begin{equation*}
\mathscr{F}_{(s,a,q)}^{\nabla}(\tau):=\Bigl( L_{(s,a)}, \ \mathcal{F}_{(s,a,q)}^{\nabla}(\tau) \Bigr).
\end{equation*}
For these objects $\mathscr{F}_{(s,a,q)}^{\nabla}(\tau)$, the analogues of Proposition \ref{generalization} and Proposition \ref{generalizationmirror} hold. Therefore, we can also regard $\mathscr{F}_{(s,a,q)}^{\nabla}(\tau)$ as a mirror of $E_{(s,a,q)}^{\nabla}(\tau)$ in this sense.

Moreover, we define two $id$-twisted smooth complex line bundles with twisted connections $\mathcal{E}_{\tau}^{\nabla}$, $\mathcal{O}_{\tau}^{\nabla}\to \check{X}_{\mathcal{G}_{\tau}^{\nabla}}^n$ as follows. Let $\mathcal{E}_{\tau}^{\nabla}\to \check{X}_{\mathcal{G}_{\tau}^{\nabla}}^n=\check{X}^n(\tau_{\mathbb{Z}})_{\check{\mathcal{G}}_{\tau_{\rm frac}}^{\nabla}}$ be an $id$-twisted smooth complex line bundle which is defined by transition functions expressed locally as 
\begin{equation*}
\psi \left( \begin{array}{cccccccc} x^1 \\ \vdots \\ x^j+1 \\ \vdots \\ x^n \\ y^1 \\ \vdots \\ y^n \end{array} \right)=\mathrm{exp}\left( \pi \mathbf{i} (\tau_{\mathbb{Z}})_j^t \check{x} \right) \cdot \psi \left( \begin{array}{cccccccc} x^1 \\ \vdots \\ x^j \\ \vdots \\ x^n \\ y^1 \\ \vdots \\ y^n \end{array} \right), \ \ \ \psi \left( \begin{array}{cccccccc} x^1 \\ \vdots \\ x^n \\ y^1 \\ \vdots \\ y^k+1 \\ \vdots \\ y^n \end{array} \right)=1 \cdot \psi \left( \begin{array}{cccccccc} x^1 \\ \vdots \\ x^n \\ y^1 \\ \vdots \\ y^k \\ \vdots \\ y^n \end{array} \right),
\end{equation*}
for each $j$, $k=1$, $\cdots$, $n$, where $\psi$ is a locally defined smooth section of $\mathcal{E}_{\tau}^{\nabla}$. We further consider a twisted connection whose connection 1-form is expressed locally as
\begin{equation*}
-\pi \mathbf{i}\check{x}^t \tau_{\mathbb{Z}} d\check{x}-\pi \mathbf{i}\check{x}^t \tau_{\rm frac} d\check{x}
\end{equation*}
on $\mathcal{E}_{\tau}^{\nabla}$. On the other hand, let $\mathcal{O}_{\tau}^{\nabla}\to \check{X}_{\mathcal{G}_{\tau}^{\nabla}}=\check{X}_{\check{\mathcal{G}}_{\tau}^{\nabla}}^n$ be the $id$-twisted trivial complex line bundle with a twisted flat connection whose connection 1-form is expressed locally as
\begin{equation*}
-\pi \mathbf{i}\check{x}^t \tau d\check{x}.
\end{equation*} 
By the definitions of $\mathcal{E}_{\tau}^{\nabla}$ and $\mathcal{O}_{\tau}^{\nabla}$, we can regard $\mathcal{L}_{(s,a,q)}^{\nabla}(\tau)$ and $\mathcal{F}_{(s,a,q)}^{\nabla}(\tau)$ as
\begin{align}
&\mathcal{L}_{(s,a,q)}^{\nabla}(\tau)\cong \Bigl. \mathcal{E}_{\tau}^{\nabla} \Bigr|_{L_{(s,a)}}\otimes \mathcal{L}_{(s,a,q)}^{\nabla}, \label{isomL} \\
&\mathcal{F}_{(s,a,q)}^{\nabla}(\tau)\cong \Bigl. \mathcal{O}_{\tau}^{\nabla} \Bigr|_{L_{(s,a)}}\otimes \mathcal{L}_{(s,a,q)}^{\nabla}, \label{isomF}
\end{align}
respectively.

Hereafter, for simplicity, we assume that a Lagrangian submanifold $L_{(s^1,a^1)}$ intersects transversally with a Lagrangian submanifold $L_{(s^2,a^2)}$. Let us quickly recall the definition of the space of morphisms of $Fuk_{\rm sub}(\check{X}^n)$. Roughly speaking, for any two objects $\mathscr{L}_{(s^1,a^1,q^1)}^{\nabla}$ and $\mathscr{L}_{(s^2,a^2,q^2)}^{\nabla}$, it is given by the vector space
\begin{equation*}
\mathrm{Hom}_{Fuk_{\rm sub}(\check{X}^n)}\left( \mathscr{L}_{(s^1,a^1,q^1)}^{\nabla},\mathscr{L}_{(s^2,a^2,q^2)}^{\nabla} \right)=\bigoplus_{p\in L_{(s^1,a^1)}\cap L_{(s^2,a^2)}}\mathrm{Hom}_{\mathbb{C}}\left( \left. \mathcal{L}_{(s^1,a^1,q^1)}^{\nabla} \right|_p,\left. \mathcal{L}_{(s^2,a^2,q^2)}^{\nabla} \right|_p \right)
\end{equation*}
over $\mathbb{C}$. Now, we assume that both $A_{\infty}$-categories $Fuk^{\rm tw}(\check{X}^n(\tau_{\mathbb{Z}})_{\check{\mathcal{G}}_{\tau_{\rm frac}}^{\nabla}})$ consisting of objects $\mathscr{L}_{(s,a,q)}^{\nabla}(\tau)$ and $Fuk^{\rm tw}(\check{X}_{\check{\mathcal{G}}_{\tau}^{\nabla}}^n)$ consisting of objects $\mathscr{F}_{(s,a,q)}^{\nabla}(\tau)$ can be formulated as analogues of $Fuk_{\rm sub}(\check{X}^n)$. Then, it is natural to consider that the spaces of morphisms of them are given as follows:
\begin{align*}
&\mathrm{Hom}_{Fuk^{\rm tw}(\check{X}^n(\tau_{\mathbb{Z}})_{\check{\mathcal{G}}_{\tau_{\rm frac}}^{\nabla}})}\left( \mathscr{L}_{(s^1,a^1,q^1)}^{\nabla}(\tau),\mathscr{L}_{(s^2,a^2,q^2)}^{\nabla}(\tau) \right) \\
&=\bigoplus_{p\in L_{(s^1,a^1)}\cap L_{(s^2,a^2)}}\mathrm{Hom}_{\mathbb{C}}\left( \left. \mathcal{L}_{(s^1,a^1,q^1)}^{\nabla}(\tau) \right|_p,\left. \mathcal{L}_{(s^2,a^2,q^2)}^{\nabla}(\tau) \right|_p \right), \\
&\mathrm{Hom}_{Fuk^{\rm tw}(\check{X}_{\check{\mathcal{G}}_{\tau}^{\nabla}}^n)}\left( \mathscr{F}_{(s^1,a^1,q^1)}^{\nabla}(\tau),\mathscr{F}_{(s^2,a^2,q^2)}^{\nabla}(\tau) \right) \\
&=\bigoplus_{p\in L_{(s^1,a^1)}\cap L_{(s^2,a^2)}}\mathrm{Hom}_{\mathbb{C}}\left( \left. \mathcal{F}_{(s^1,a^1,q^1)}^{\nabla}(\tau) \right|_p,\left. \mathcal{F}_{(s^2,a^2,q^2)}^{\nabla}(\tau) \right|_p \right).
\end{align*}
In particular, via the isomorphisms (\ref{isomL}), (\ref{isomF}), we have
\begin{align*}
&\mathrm{Hom}_{Fuk^{\rm tw}(\check{X}^n(\tau_{\mathbb{Z}})_{\check{\mathcal{G}}_{\tau_{\rm frac}}^{\nabla}})}\left( \mathscr{L}_{(s^1,a^1,q^1)}^{\nabla}(\tau),\mathscr{L}_{(s^2,a^2,q^2)}^{\nabla}(\tau) \right)\cong \mathrm{Hom}_{Fuk_{\rm sub}(\check{X}^n)}\left( \mathscr{L}_{(s^1,a^1,q^1)}^{\nabla},\mathscr{L}_{(s^2,a^2,q^2)}^{\nabla} \right), \\
&\mathrm{Hom}_{Fuk^{\rm tw}(\check{X}_{\check{\mathcal{G}}_{\tau}^{\nabla}}^n)}\left( \mathscr{F}_{(s^1,a^1,q^1)}^{\nabla}(\tau),\mathscr{F}_{(s^2,a^2,q^2)}^{\nabla}(\tau) \right)\cong \mathrm{Hom}_{Fuk_{\rm sub}(\check{X}^n)}\left( \mathscr{L}_{(s^1,a^1,q^1)}^{\nabla},\mathscr{L}_{(s^2,a^2,q^2)}^{\nabla} \right),
\end{align*}
and these relations indicates the identification
\begin{equation*}
\mathrm{Hom}_{Fuk^{\rm tw}(\check{X}^n(\tau_{\mathbb{Z}})_{\check{\mathcal{G}}_{\tau_{\rm frac}}^{\nabla}})}\left( \mathscr{L}_{(s^1,a^1,q^1)}^{\nabla}(\tau),\mathscr{L}_{(s^2,a^2,q^2)}^{\nabla}(\tau) \right)\cong \mathrm{Hom}_{Fuk^{\rm tw}(\check{X}_{\check{\mathcal{G}}_{\tau}^{\nabla}}^n)}\left( \mathscr{F}_{(s^1,a^1,q^1)}^{\nabla}(\tau),\mathscr{F}_{(s^2,a^2,q^2)}^{\nabla}(\tau) \right).
\end{equation*}
Thus, we can expect that there exists the equivalence
\begin{equation*}
Fuk^{\rm tw}(\check{X}^n(\tau_{\mathbb{Z}})_{\check{\mathcal{G}}_{\tau_{\rm frac}}^{\nabla}})\cong Fuk^{\rm tw}(\check{X}_{\check{\mathcal{G}}_{\tau}^{\nabla}}^n)
\end{equation*}
as $A_{\infty}$-categories which is obtained by considering the correspondence
\begin{equation*}
\mathscr{L}_{(s,a,q)}^{\nabla}(\tau) \longmapsto \mathscr{F}_{(s,a,q)}^{\nabla}(\tau).
\end{equation*}
This implies the following: it can be expected that our construction does not depend on the choice of the divisions $\tau=\tau_{\mathbb{Z}}+\tau_{\rm frac}=O+\tau$ if we can formulate ``$Fuk^{\rm tw}(\check{X}^n(\tau_{\mathbb{Z}})_{\check{\mathcal{G}}_{\tau_{\rm frac}}^{\nabla}})$'' and ``$Fuk^{\rm tw}(\check{X}_{\check{\mathcal{G}}_{\tau}^{\nabla}}^n)$'' as analogues of $Fuk_{\rm sub}(\check{X}^n)$. Also, we can regard ``the equivalence'' 
\begin{equation*}
Fuk_{\rm sub}(\check{X}^n)\stackrel{\sim}{\to} Fuk^{\rm tw}(\check{X}^n(\tau_{\mathbb{Z}})_{\check{\mathcal{G}}_{\tau_{\rm frac}}^{\nabla}})
\end{equation*}
which is induced by the correspondence
\begin{equation*}
\Bigl( L_{(s,a)}, \ \mathcal{L}_{(s,a,q)}^{\nabla} \Bigr) \longmapsto \Bigl( L_{(s,a)}, \ \Bigl. \mathcal{E}_{\tau}^{\nabla} \Bigr|_{L_{(s,a)}}\otimes \mathcal{L}_{(s,a,q)}^{\nabla} \Bigr)
\end{equation*}
as a mirror dual functor of the equivalence (\ref{dgequiv}).

\section*{Acknowledgment}
I would like to thank the referee for reading this paper carefully.

\end{document}